\documentclass[reqno]{amsart}
\usepackage{amsmath,amssymb,latexsym,epsfig, url,graphicx,newtxtext,newtxmath,tikz}
 
\newtheorem{conjecture}{Conjecture}
\newtheorem{corollary}{Corollary} 
\newtheorem{definition}{Definition} 
 
\newtheorem{proposition}{Proposition}

\newtheorem{theorem}{Theorem}

\begin{document}

\title[Higher-Order Derivatives of First-Passage Percolation]{Higher-Order Derivatives of First-Passage Percolation with Respect to the Environment} 

\author[I. Mati\'c]{
Ivan Mati\'c}
\address{
Baruch College, City University of New York}\email{ivan.matic@baruch.cuny.edu}
\author[R. Radoi\v ci\'c]{
Rado\v s Radoi\v ci\'c} \address{
Baruch College, City University of New York}\email{rados.radoicic@baruch.cuny.edu} 
\author[D. Stefanica]{
Dan Stefanica}\address{
Baruch College, City University of New York}\email{dan.stefanica@baruch.cuny.edu}

\begin{abstract}
We introduce and study derivatives in first-passage percolation with edge weights given by i.i.d. random variables supported on $\{a,b\}$. 
We show that the variance of the passage time can be expressed in terms of these derivatives. We further analyze their structure and establish several 
fundamental properties and bounds. Our bounds for the lower Fourier levels on the torus model raise the prospect that, in dimensions 3 and higher, the variance may grow slower than any positive power of $n$. Such growth would contradict the commonly held belief that the fluctuation exponent is positive. \end{abstract}

\maketitle

\section{Introduction}

\subsection{Definition of the model}
\noindent The first-passage percolation model was introduced by Hammersley and Welsh in \cite{HammersleyWelsh1965}. Let $p\in (0,1)$ and let $a$ and $b$ be two positive real numbers that satisfy $a<b$. 
Consider the graph whose vertices are elements of $\mathbb Z^d\cap [-2n,2n]^d$ (with $d\geq 2$), where two vertices $(x_1, \dots, x_d)$ and $(y_1,\dots, y_d)$ are connected by an edge if $|x_1-y_1|+\cdots+|x_d-y_d|=1$. 

Let $W_n$ denote the set of edges of this graph. The sample space is defined as $\Omega_n=\{a,b\}^{W_n}$. Each edge $e$ of the graph is independently assigned a passage time of either $a$ or $b$, with probabilities $\mathbb P(a)= p$ and $\mathbb P(b)=1-p$.  

For a fixed environment $\omega\in\Omega_n$ and a path $\gamma$, the passage time $T(\gamma,\omega)$ is defined as the sum of the values assigned to the edges of $\gamma$. For two fixed vertices $u$ and $v$, the passage time $f(u,v,\omega)$ is the random variable defined as the minimum of $T(\gamma,\omega)$ over all paths $\gamma$ connecting $u$ to $v$.

When $v$ is fixed, we will also use the notation $f_n(\omega)$, $f_n$, or simply $f$, in place of $f(0,nv,\omega)$.

For some of our results, we will only work on a simplified model from \cite{benjamini2006} that has more symmetry and fewer technical challenges: the percolation is considered on the torus $\mathbb Z_{n}^d$. We will use superscript $\tau$ and write $f_n^{\tau}(\omega)$, $f_n^{\tau}$, or $f^{\tau}$ to emphasize when we are working on this simplified torus model. 
Formally, a $d$-dimensional torus is a graph whose vertices are elements of $\mathbb Z_{n}^d$ and two vertices $u$ and $v$ are connected by an edge if and only if there is exactly one coordinate $k\in\{1,2,\dots, d\}$ such that $u_k-v_k\equiv \pm1$ (mod $n$) and $u_j=v_j$ for all $j\neq k$.
The set of admissible paths $\Gamma$ consists of all paths that wrap around the torus in the direction $x_1$. Hence, the path $(s,v_1, \dots, v_m,e)$ belongs to the set $\Gamma$ if it is a path in the graph and if the starting and ending vertices $s$ and $e$ have all the coordinates the same except for the first coordinate. Their first coordinates are $s_1=0$ and $e_1=n-1$.  The random variable $f^{\tau}_n$ is defined as the minimum of the passage times among all paths $\gamma$ that consist of adjacent edges and that wrap around the torus exactly once in the direction $x_1$.  
The function $f^{\tau}_n$ is defined as 
$
f^{\tau}_n(\omega)=\min\left\{T(\gamma,\omega): \gamma \in \Gamma
\right\}.
$

A path $\gamma$ is called {\em geodesic} if the minimum $f_n(\omega)$ (or $f^{\tau}_n$, depending on which problem we are studying) is attained at $\gamma$, i.e. if $f_n(\omega)=T(\gamma,\omega)$.

\subsection{Definition of environment derivatives}
If we denote by $W_n$ the set of all edges, then the sample space is $\Omega_n=\{a,b\}^{W_n}$. We will often omit the subscript $n$ when there is no danger of confusion. For each edge $j$ and each $\omega\in \Omega$, we define $\sigma_j^a(\omega)$ as the outcome from $\Omega$ whose passage time over the edge $j$ is changed from $\omega_j$ to $a$, regardless of what the original value $\omega_j$ was. The operation $\sigma_j^b$ is defined in an analogous way. Formally, for $\delta\in\{a,b\}$, we define $\sigma_j^{\delta}:\Omega\to\Omega$ with  
\[
\left[\sigma_j^{\delta}(\omega)\right]_k = \left\{\begin{array}{ll}\omega_k,& k\neq j,\\
\delta, &k=j.\end{array}\right.  
\]
If $\varphi:\Omega\to\mathbb R$ is any random variable, then the {\em first order environment derivative} $\partial_j\varphi$ is the random variable defined as 
\[
\partial_j\varphi=\varphi\circ \sigma_j^b-\varphi\circ \sigma_j^a.  
\]

For two distinct edges $i$ and $j$, we will give the name {\em second order environment derivative} to the quantity $\partial_{i}\partial_{j}\varphi$. In general, if $S$ is a non-empty subset of $W$, the operator $\partial_S\varphi$ is defined recursively as 
\[
\partial_S\varphi = \partial_{S\setminus\{j\}}\left(\partial_j\varphi\right), 
\]
where $j$ is an arbitrary element of $S$. The definition is independent on the choice of $j$, since a simple induction can be used to prove that for $S=\{s_1,\dots, s_m\}$, the following holds 
\[
\partial_S\varphi=\sum_{\theta_1\in\{a,b\}} \cdots\sum_{\theta_m\in\{a,b\}} (-1)^{\mathbf1_a(\theta_1)+ \cdots+\mathbf1_a(\theta_m)}\varphi\circ\sigma_{s_1}^{\theta_1} \circ\cdots\circ \sigma_{s_m}^{\theta_m}.  
\]
The function $\mathbf1_a:\{a,b\}\to \{0,1\}$ assigns the value $1$ to $a$ and $0$ to $b$.

\subsection{Variance decomposition and Fourier levels} The variances of $f_n$ and $f_n^{\tau}$ can be
written in terms of environment derivatives. 
\begin{theorem}
\label{thm|varianceFormula}
Let $f$ be a random variable on $\Omega_n$. The following equality holds
\begin{align}
\text{var}(f)&=\sum_{M\subseteq W, M\neq\emptyset}\left(p(1-p)\right)^{|M|}\left(\mathbb E\left[\partial_Mf\right]\right)^2. \label{eqn|varianceFormula}
\end{align}
\end{theorem}

Let us define Fourier level $\Sigma_j$ as 
\begin{align*}
\Sigma_j(f)
&=
\sum_{\substack{M\subseteq W,\;|M|=j}}
\left(\mathbb{E}\big[\partial_M f\big]\right)^2.
\end{align*}

Then, for each $k$, we can break \eqref{eqn|varianceFormula} into the {\em lower sum} $L_k$ and the {\em remainder} $R_k$
\begin{align*}
\text{var}(f)&=L_k(f)+R_k(f),\quad\text{where} \\
\quad L_k(f)&=\sum_{j=1}^{k-1}(p(1-p))^j\Sigma_j(f)\; \text{ and }\; R_k(f)=\sum_{j=k}^{\infty}(p(1-p))^{j}\Sigma_{j}(f). 
\end{align*} 
 
\subsection{Bound on the lower sum $L_k$}\label{subsec|BLower}
On the torus model, we can prove that for fixed $k$ and fixed $\zeta\in(0,1)$, the lower sum $L_k(f^{\tau})$ is at most $O( n^{1+\zeta-\zeta d})$. 
In dimensions $d \geq 3$, choosing $\zeta > 1/(d-1)$ yields $1 - \zeta(d-1) < 0$, 
so the lower sums converge to $0$. We believe this also holds for general model, not only torus model. We also believe this holds for $\zeta=1$. The case $\zeta<1$ can be handled by a general functional analysis technique that relies on Beckner-Bonami inequality. The case $\zeta=1$ must be handled by combinatorics. We were able to extend the bounds for $\Sigma_1$ and $\Sigma_2$ to $\zeta=1$ \cite{SecondPaper2026}, but a more comprehensive combinatorial analysis would be needed for higher levels. Below is the detailed formulation of the theorem that we will prove in this paper.

\begin{theorem}\label{thm|SumLowerFourierLevels}
For every $n\in \mathbb N$, every $k\in\mathbb N$, and every $\zeta\in(0,1)$, the following holds
\begin{align}
L_k(f^{\tau})&\leq C(\zeta,k,p,a,b)\cdot n^{1+\zeta-\zeta d},\quad\text{ where }\label{eqn|SumLowerFourierLevels}\\
C(\zeta,k,p,a,b) &=p(1-p)\cdot (b-a)^2\cdot \left(\frac{b}{ap}\right)^{1+\zeta} \nonumber\\
&\quad\times\frac{\left(\frac{1+\zeta}{(1-\zeta)p(1-p)}\right)^{k-1}-1}{\frac{1+\zeta}{(1-\zeta)p(1-p)}-1}.\label{eqn|CoefficientInFrontOfSumOfLevels}
\end{align}
\end{theorem}

\subsection{Bound on the remainder $R_k$} \label{subsec|BRemainder}
 The following result holds for general random variables. Hence, it applies both to $f$ and $f^{\tau}$.
\begin{theorem}\label{thm|BoundRemainder}
 For every integer $k\geq 1$, there exists a real constant $C$ and an integer $n_0$ such that 
for $n\geq n_0$, and every random variable $f$ on $\Omega_n$ its remainder $R_k=R_k(f)$ satisfies
\begin{align}
R_k(f)&\leq 
C\cdot \sum_{M\subseteq W, |M|=k, \|\partial_Mf\|_1\neq 0}\frac{ \left\|\partial_Mf\right\|_2^2}{
1+\left(\log\frac{\|\partial_Mf\|_2}{\|\partial_Mf\|_1}\right)^k
},
\label{eqn|BoundRemainder}
\end{align} 
where $\|g\|_p$ is the $L^p$-norm of the function $g$ defined as \[\|g\|_p=\left(\int |g|^p\,d\mathbb P\right)^{1/p}=\left(\mathbb E[|g|^p]\right)^{1/p}.\]
\end{theorem}

\subsection{Conjectures about variance in higher dimensions}
Theorem~\ref{thm|SumLowerFourierLevels} casts the variance 
conjecture in a new light, not by bringing it closer to resolution, 
but by making it more mysterious. In dimensions $d\geq 3$, the lowest Fourier levels contribute at most $O( n^{1+\zeta-\zeta d}) $ to the variance. The value $\zeta$ can be chosen to make the exponent of $n$ negative. If the 
widely conjectured bound $\mathrm{var}(f)=O(n^{2\chi})$ with 
$\chi>0$ holds in dimensions greater than or equal to $3$, then the divergence cannot come from the lowest levels. It must come entirely from higher Fourier 
levels. This would be a striking phenomenon: the cumulative effect of higher-order interactions produces 
polynomial growth. On the other hand, if higher Fourier levels are 
also well-controlled, as we conjecture, particularly in dimensions 
$d\geq 3$ where the $L^2$ norms $\|\partial_M f\|_2$ may decay 
exponentially in $k$, then the variance itself would be $O(1)$, 
implying $\chi=0$ in higher dimensions. Either conclusion would be 
remarkable, and both remain out of reach.

\subsection{Methods used in proofs and overview of literature}

The inequality \eqref{eqn|BoundRemainder} in the case $k=1$ was proved by Talagrand \cite{talagrand1994}. In \cite{Tanguy2018}, Tanguy generalized Talagrand's inequality to include second order environment derivatives. In \cite{Tanguy2018}, the second order result was
\begin{align}
\text{var}(f)&\leq C\left(\sum_{i} \|\partial_if\|_{1+e^{-2s_0}}^2 + 
\sum_{M\subseteq W, |M|=2, \|\partial_Mf\|_1\neq 0}\frac{ \left\|\partial_Mf\right\|_2^2}{
1+\left(\log\frac{\|\partial_Mf\|_2}{\|\partial_Mf\|_1}\right)^2
}\right).
\label{eqn|Tanguy}
\end{align}
The bound \eqref{eqn|Tanguy} treats the $k=2$ case for Boolean functions on the uniform discrete cube 
$\{-1,1\}^n$ via the Bonami-Beckner semigroup. Tanguy \cite{Tanguy2018} remarked that the proof technique could be extended to general $k$, while noting that the notation would become heavy. Further iteration of the semigroup machinery would likely introduce additional complexities in the first summation.

The results of Talagrand \cite{talagrand1994} and Benjamini, Kalai, and Schramm \cite{benjamini2006}
show how inequalities with first order environment derivatives lead to bounds on the variance of the form $\text{var}(f_n)\leq C\cdot \frac{n}{\log n}$ and $\text{var}(f^{\tau}_n)\leq C\cdot\frac{n}{\log n}$, for some constant $C$.

In dimension $d = 1$, the right-hand side of \eqref{eqn|SumLowerFourierLevels} becomes $O(n)$, which 
is consistent with the central limit theorem: the passage time $T(0,ne_1)$ 
is a sum of $n$ i.i.d.\ variables with variance $\Theta(n)$, and the bound 
is trivially satisfied. The bound becomes informative only in $d \geq 2$, 
where the FPP variance is conjectured to be sublinear in $n$.

The conjectured upper bound for the variance in first-passage percolation is $C\cdot n^{2\chi}$, where $\chi$ is an exponent that depends on the dimension. 
Current predictions suggest that $\chi$ is $\frac13$ in the two-dimensional case \cite{AuffingerDamronHanson2017}. 
A bound of $\chi\leq\frac12$ was established by Kesten in 1993 \cite{Kesten1993}, and to date, there is no formal proof that $\chi$ is strictly less than $\frac12$. 
In dimensions higher than two, even conjectural values for $\chi$ remain unclear. According to \cite{AuffingerDamronHanson2017}, there are reasons to believe that $\chi$ 
remains strictly positive in all dimensions, though it tends to $0$ as $d\to\infty$. 
The existence and location of a critical dimension, above which the variance ceases to grow as $n^{\chi}$, is studied in the probability literature \cite{AlexanderFPPCriticalDim2023} and predicted in the physics literature \cite{Fogedby2006}. 
Our bound \eqref{eqn|SumLowerFourierLevels} is consistent with this behavior: in dimension $3$ and higher, we can choose $\zeta$ for which the exponent of $n$ becomes negative. 
The coefficient $C(\zeta,k,p,a,b)$ in \eqref{eqn|CoefficientInFrontOfSumOfLevels} is exponential in $k$ and goes to infinity when $\zeta$ approaches $1$.

It is worth noting that the value $\chi=\frac13$ was rigorously established by Johansson in a related model known as the totally asymmetric simple exclusion process (TASEP) \cite{Johansson2000}. 
The TASEP model belongs to a class of exactly solvable models that can be analyzed using techniques from random matrix theory. In this setting, a central limit theorem has been proven, with the limiting 
distribution given by the Tracy-Widom distribution for the
largest eigenvalue \cite{TracyWidom1994}.

The best current variance bound for first-passage percolation is $C\cdot\frac{n}{\log n}$. It is obtained by Benjamini, Kalai, and Schramm \cite{benjamini2006}. 
Their approach relies on Talagrand's inequality \cite{talagrand1994}. After applying the inequality, they use symmetries of the first-passage percolation models. 

Our Theorem \ref{thm|BoundRemainder} generalizes Talagrand's inequality in the sense that the latter becomes a special case when $k=1$. In the torus model and the case $k=2$, it is possible to improve the denominator 
to $(\log n)^2$ at the cost of introducing the term $\mathbb E\left[N_2\right]$ in the numerator. More precisely,

\begin{corollary}\label{thm|Consequence2} The pair of edges $(i,j)$ is called {\em convoluted} on the outcome $\omega\in \Omega$, if $\partial_i\partial_jf(\omega)\neq 0$.   
If $N_2$ is the random variable that represents number of convoluted pairs of edges, then there exist constants $C$, $\hat C$, $C'$ and an integer $n_0$ such that for $n\geq n_0$, the variance of $f^{\tau}$ on torus satisfies 
\begin{align}
\text{var}(f^{\tau}) &\leq C\frac{\sum_{|M|=2}\|\partial_Mf^{\tau}\|_2^2}{(\log n)^2}+C'\leq \hat C\frac{\mathbb E[N_2]}{(\log n)^2}+C'.  \label{eqn|convN} 
\end{align}
\end{corollary}

Currently, we are unable bound 
$\mathbb E[N_2]$ by $n$; hence, our result does not improve upon the best-known bound of $\frac{n}{\log n}$.   
We conjecture that $\mathbb E[N_2]$ and the $L^2$ norms $\|\partial_M f\|_2$ are small--especially in dimensions $d\geq 3$, where the decay could potentially be exponentially fast.
However, these quantities remain difficult to analyze at present.

A complete understanding of the environment derivatives is equivalent to a complete understanding of the variance, due to \eqref{eqn|varianceFormula}. 
Theorem \ref{thm|varianceFormula} is equality. Together with 
$\text{var}(f)\leq C\frac n{\log n}$, 
 it can be used to derive certain $L^2$-bounds on environment derivatives. 
We will list in 
Section \ref{sec|L2Bounds} some conjectures that would be sufficient for algebraic improvements on the variance bound.  
The equation \eqref{eqn|varianceFormula} is not particularly surprising--it is the Parseval’s identity for the Fourier expansion. 
Talagrand, as well as Benjamini, Kalai, and Schramm, have previously employed the Fourier expansion of variance, but skillfully avoided dealing with the coefficients directly by relying on 
clever bounding techniques. Similar techniques 
and different forms of the variance equations were developed in \cite{Tal2017} and \cite{Przybylowski2024} for Boolean circuits and functions.
We believe that there is a special value of 
\eqref{eqn|varianceFormula} because it expresses the coefficients in terms of environment derivatives.   
 
In summary, the ultimate goal is to control the $L^2$-norms of the environment derivatives $\partial_M f$, as these norms are directly tied to the variance. At present, however, we are unable 
to effectively bound these $L^2$-norms. 

This paper makes progress in analyzing the environment derivatives. We introduce the concepts of essential and influential edges and study the relationship between these categories of edges. The results that follow make it possible to further clarify the anomalous changes of the geodesic that can happen with minimal changes in the environments. One of major benefits is the ability to
bypass intuitive representations--representations that become increasingly difficult to construct when the sets 
$M$ contain more than a few elements.

The fundamental results about essential and influential edges that were derived in this paper paved a way for further study of environment derivatives and obtain almost sure bounds, \cite{SecondPaper2026}. The derivatives of order $1$ are obviously bounded by $(b-a)$ from above and $0$ from below.
For $k\geq 2$, it is possible to construct examples where the environment derivative of order $k$ are equal to $\binom{k-2}{\lceil\frac{k-2}2\rceil}(b-a)$ and examples where the derivatives are 
$-\binom{k-2}{\lceil\frac{k-2}2\rceil}(b-a)$. 
For $k\in\{2,3,4\}$ these binomial coefficients turn out to be $1$, $1$, and $2$. 
We can prove that for $k\in\{1,2,3,4\}$ these examples are in some sense the worst case scenarios, i.e. that the environment derivatives are bounded by the above binomial coefficients.

We conjecture that these binomial coefficients are the actual bounds for all $k$, but we can't prove this at the moment. We have an ongoing work at building computer-assisted proofs for almost sure bounds of higher order.

The proof of Theorem \ref{thm|BoundRemainder} relies on the Beckner–Bonami inequality from \cite{beckner1975} and \cite{Bonami1970}, similar to Talagrand’s original approach. In our proof, we clearly separate probabilistic components from algebraic manipulations and extend the variance decomposition to gain higher powers in the denominator. The logarithmic improvement in the denominator is more transparent in our presentation due to this clearer separation between probability and algebra.

We modified Talagrand's method by generalizing his operator $\Delta_i$ (denoted $\rho_i$ in \cite{benjamini2006}). Talagrand's operator is defined as 
\[\Delta_if(\omega)=f(\sigma_i(\omega))-f(\omega),\]
 where $\sigma_i(\omega)$ denotes the environment in which the passage time over edge $i$ is changed from its original value. 
 Our first-order environment derivative $\partial_i$ is defined as \[\partial_if(\omega)=f(\sigma_i^b(\omega))-f(\sigma_i^a(\omega)).\] 
 This seemingly small change leads to significant improvements in clarity, particularly in identifying edges that belong to geodesics and those for which the environment derivatives are nonzero. In addition, the integration-by-parts formulas become much simpler with the operator $\partial_i$, as it is a more natural extension of the classical derivative than $\Delta_i$.  If the denominator $(b-a)$ were introduced to normalize the derivative, the ordering of terms would align with the numerator $f\circ\sigma_i^b-f\circ\sigma_i^a$ of $\partial_if$. 
 
We generalize this environment derivative to higher orders, which allows us to make a tradeoff after applying the Beckner–Bonami inequality. This tradeoff improves the denominator from $\log n$ to $(\log n)^k$, 
but at the cost of introducing the $L^2$-norms $\|\partial_Mf\|_2^2$ of the higher-order environment derivatives into the numerator. As mentioned earlier, these $L^2$-norms are not easy to control. 
 We hope that other researchers will explore the theory of environment derivatives further, as they show promise for deeper understanding and improved bounds.
 
 When the edge passage times are supported on $\{a,b\}$, as in the model studied in this paper, geodesics may not be unique. It is expected that there will be numerous sufficiently disjoint geodesics, which would imply a small number of influential edges. This, in turn, could make it easier to obtain bounds on $N_2$. The study of geodesics has produced several important results and highly credible conjectures. Notably, as the size of the environment grows, at least two infinite geodesics are expected to emerge \cite{Hoffman2008}. Infinite geodesics are also known to coalesce with high probability \cite{Alexander2023},  \cite{Seppalainen2020}, \cite{KrishnanRassoulAghaSeppalainen2023}. 

If the edge passage times are continuously distributed, geodesics are unique, and the event $A_i=\{\partial_if\neq 0\}$ 
coincides with the event that edge $i$ is essential--that is that is, every geodesic passes through $i$.  
Benjamini, Kalai, and Schramm studied the discrete case and encountered a major challenge: proving that the probability of $A_i$ decays as $n^{-\xi}$. 
If the event $A_i$ occurs, we say that the edge $i$ is influential. The authors of \cite{benjamini2006} proposed a simpler problem: prove that $\mathbb P(A_i)\to 0$. 
This problem was resolved recently. 
In the continuous setting, we now have bounds of the form $\mathbb P(A_j)\leq Cn^{-\xi}$. 
The first such results appeared in \cite{DamronHanson2017}, were strengthened in \cite{AhlbergHoffman2019} (which removed differentiability assumptions), and culminated in polynomial bounds in \cite{DembinElboimPeled2024}.

 Over the past 20 years, the method developed by Benjamini, Kalai, and Schramm has been successfully used to bound variances in numerous problems, many now categorized as superconcentration problems \cite{chatterjee2014} or part of the Kardar–Parisi–Zhang (KPZ) universality class \cite{AlbertsKhaninQuastel2014}, \cite{CorwinGhosalHammond2021}. First-passage percolation models can also be viewed as extreme cases of random polymers in the zero-temperature limit \cite{Zygouras2024}.

In \cite{benaim2008} and \cite{DamronHansonSosoe2015}, the $\frac{n}{\log n}$ variance bound was extended to a large class of distributions. The exponent $\chi$, discussed earlier, is called the fluctuation exponent. It is related to the transversal exponent $\xi$, defined as the number for which $C\cdot n^{\xi}$ is the maximal distance from the geodesic to the straight line between the starting and ending point. 
The exponents $\chi$ and $\xi$ satisfy the KPZ scaling relation $\chi=2\xi-1$. The inequality $\chi\geq 2\xi-1$ was proved in \cite{NewmanPiza1995}, while the reverse inequality $\chi\leq 2\xi-1$ was first 
shown in \cite{chatterjee2013},  then generalized and simplified in \cite{AuffingerDamron2014}. These scaling exponents are closely tied to the asymptotic shape of the balls in the first-passage percolation metric; see \cite{CoxDurrett1981} and \cite{ChatterjeeDey2016}.

The models we study in this paper are discrete. However, there have been successful generalizations to models where graphs consist of points scattered in Euclidean space \cite{HowardNewman1997}. Scaling relations and large deviation estimates have been established for both these spatial models and traditional lattice models in \cite{BasuGangulySly2021} and \cite{BasuSidoraviciusSly2023}. In a broad class of first-passage percolation models, the limit shape has been shown to be differentiable \cite{BakhtinDow2024}.
These problems become even more continuous when framed in terms of random Hamilton-Jacobi equations. For generalizations of the law of large numbers and central limit theorems in this context, see \cite{RezakhanlouTarver2000}, \cite{ArmstrongCardaliaguetSouganidis2014}, and \cite{DaviniKosyginaYilmaz2023}. Variance bounds of the order $\frac{n}{\log n}$ have also been obtained in this continuous PDE setting in \cite{MaticNolen2012}.

\section{Essential and influential edges}
We will distinguish four categories to which an edge of the graph can belong. These categories may overlap but are conceptually distinct. Most edges will not belong to any of them. 

\begin{definition} \label{def|Essential} An edge $j\in W_n$ is called {\em essential} on the environment $\omega$ if every geodesic passes through $j$. We will denote by $E_j$ the event that the edge $j$ is essential. 
\end{definition}

\begin{definition} \label{def|SemiEssential} An edge $j\in W_n$ is called {\em semi-essential} on the environment $\omega$ if at least one geodesic passes through $j$. We will denote by $\hat E_j$ the event that the edge $j$ is semi-essential. 
\end{definition}

\begin{definition} \label{def|Influential} An edge $j\in W_n$ is called {\em influential} if $\partial_jf(\omega)\neq 0$. We will denote by $A_j$ the event that the edge $j$ is influential. 
\end{definition}

\begin{definition} \label{def|VeryInfluential} An edge $j\in W_n$ is called {\em very influential} if $\partial_jf(\omega)= b-a$. We will denote by $\hat A_j$ the event that the edge $j$ is very influential. 
\end{definition} 

Since the passage times across edges have a discrete distribution, there may be multiple geodesics between two fixed endpoints. We will show that, in general, the four categories defined above are distinct. Later in this section, we will prove that the general relationship between the events $A_j$, $\hat A_j$, $E_j$, and $\hat E_j$ is captured by the Venn diagram in Figure \ref{fig|relationshipBetweenAE}. 

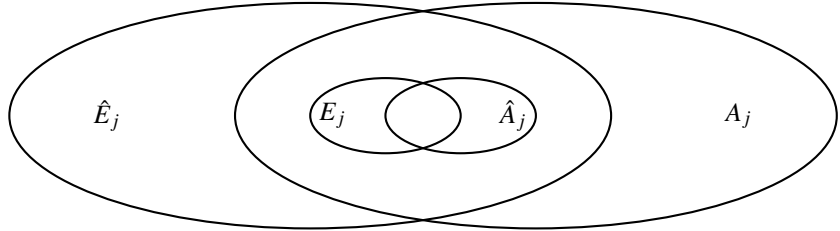
\begin{figure}[t]
\centering
\begin{tikzpicture}[scale=0.6]
\draw[thick] (2.6,2) ellipse (2.6 and 1.5);
\draw[thick] (4.6,2) ellipse (2.6 and 1.5);
\draw[thick] (3.20,1.9) ellipse (0.6 and 0.3);
\draw[thick] (4.00,1.9) ellipse (0.6 and 0.3);
\node[right] at (0.00,3.4) {$\hat E_j$};
\node[left]  at (7.25,3.4) {$A_j$};
\node[right] at (2.50,2.6) {$E_j$};
\node[left]  at (4.70,2.6) {$\hat A_j$};
\node at (3.6,0) {(a) general case};

\draw[thick] (11.6,2) ellipse (3.5 and 1.5);
\draw[thick] (12.7,2) ellipse (1.8 and 0.85);
\draw[thick] (13.0,2) ellipse (0.7 and 0.35);
\node[right] at (8.3,3.4) {$\hat E_j$};
\node at (10.3,2.7) {$A_j{=}\hat A_j$};
\node at (12.0,2.0)  {$E_j$};
\node at (11.6,0) {(b) $a/(b{-}a)\in\mathbb N$};
\end{tikzpicture}
\caption{Relationship between essential, semi-essential, influential, and very influential sets. Left: general case. Right: when $a/(b-a)\in\mathbb N$, we have $A_j=\hat A_j$, so the four events collapse into three nested ones.}
\label{fig|relationshipBetweenAE}
\end{figure}

The inclusion $E_j\subseteq A_j$ is the most important of all of the inclusions from the diagram. Although $E_j$ is subset of $A_j$, we will prove later in Theorem \ref{thm|BoundMVGeneral} that 
the two events have comparable probabilities, i.e. $\mathbb P(A_j)\leq \mathbb P(E_j)/p$. It is natural to conjecture that all of the events $E_j$, $A_j$, $\hat E_j$ and $\hat A_j$ have comparable probabilities. We didn't need
this full result in our paper, and the proof does not look obvious. Here is the formal conjecture.  

\begin{conjecture} There exists a constant $C$ independent on $n$ such that 
\begin{align*}\mathbb P(A_j)&\leq C\cdot \mathbb P(\hat A_j),\\
\mathbb P(\hat E_j)&\leq C\cdot \mathbb P(E_j),
\quad\text{ and }\\ \mathbb P(\hat E_j)&\leq C\cdot \mathbb P(\hat A_j).
\end{align*}
\end{conjecture}

Most research involving percolation models, where passage times have discrete distributions, has had to address these distinct categories of edges. 
Handling this distinction has often introduced technicalities that researchers needed to overcome. This section shows the similarities and differences among the various edge categories and summarizes their relationships. 
 
The results of this section apply to both $f$ an $f^{\tau}$; we will only present them for $f$. 
We will start with a proposition whose proof we will omit because it is straightforward.
\begin{proposition}\label{thm|Straightforward} For every $i\neq j$, every $\alpha,\beta\in\{a,b\}$, and every random variable $\varphi$, 
\begin{align} 
\sigma_i^{\alpha}\circ \sigma_i^{\beta}&=\sigma_i^{\alpha};   \label{eqn|propertyA}\\
\sigma_i^{\alpha}\circ \sigma_j^{\beta}&=\sigma_j^{\beta}\circ\sigma_i^{\alpha};  \label{eqn|propertyB} \\
(\partial_i \varphi)\circ \sigma_i^{\alpha} &= \partial_i \varphi;  \label{eqn|propertyC}
\end{align} 
\begin{align}
\partial_i\partial_i \varphi&=0;  \label{eqn|propertyD} \\
\partial_i\partial_j \varphi&=\partial_j\partial_i\varphi;  \label{eqn|propertyE} \\
\varphi\cdot 1_{\omega_i=\alpha}&=\varphi\circ \sigma_i^{\alpha}\cdot 1_{\omega_i=\alpha}.  \label{eqn|propertyF}
\end{align}
\end{proposition}

\noindent{\bf Remark.} Here is an immediate consequence of 
\eqref{eqn|propertyC}: for every function $\varphi$, every 
$\xi\in\{a,b\}$, and every $B\subseteq \mathbb R$,
\begin{align} 
\left(\partial_i\varphi\right)^{-1}(B)
=\left((\partial_i\varphi) \circ \sigma_i^{\xi}\right)^{-1}(B)
=\left(\sigma_i^{\xi}\right)^{-1}\!\left((\partial_i\varphi)^{-1}(B)\right).
\label{eqn|invarianceLevelSetsAlgebraicGeneralForm}
\end{align}
We now apply \eqref{eqn|invarianceLevelSetsAlgebraicGeneralForm} to $\varphi=f$ in two special cases. 
Taking $B=\mathbb R\setminus\{0\}$ gives the first equality below, and taking $B=\{b-a\}$ gives the second. 
For every edge $i$ and every $\xi\in\{a,b\}$, 
\begin{align*}&
\left(\sigma_i^{\xi}\right)^{-1}(A_i)=A_i 
\quad\text{and}\quad 
\left(\sigma_i^{\xi}\right)^{-1}(\hat A_i)=\hat A_i. 
\end{align*}

The next theorem is one of the few results that require a combinatorial analysis of geodesics. The obtained algebraic relationships among $E_j$, $A_j$, $\hat E_j$, and $\hat A_j$, together with some general results from set theory, imply all inclusions in Figure \ref{fig|relationshipBetweenAE} (a).

\begin{theorem}\label{thm|AiDefinition} The events $E_j$, $A_j$, $\hat E_j$, and $\hat A_j$ satisfy
\begin{align}
 A_j&=(\sigma_j^a)^{-1}\left(E_j\right); \label{eqn|AiDefinition}\\
 \hat A_j&=(\sigma_j^b)^{-1}(\hat E_j). \label{eqn|HatAiDefinition}
\end{align}
\end{theorem}
\begin{proof} 
If we assume $ \sigma_j^a (\omega)\not \in E_j$, then there is a geodesic $\gamma$ on $\sigma_j^a(\omega)$ that does not pass through $j$. The value $f(\sigma_j^a(\omega))$ satisfies 
\begin{align*}
f(\sigma_j^a(\omega))&=T(\gamma,\sigma_j^a(\omega))=T(\gamma,\sigma_j^b(\omega))\geq f(\sigma_j^b(\omega)).
\end{align*} The monotonicity of $f$ in each coordinate implies $f(\sigma_j^a(\omega))\leq f(\sigma_j^b(\omega))$. We obtained $f(\sigma_j^a(\omega))= f(\sigma_j^b(\omega))$, and we can conclude that
 $\partial_jf(\omega)\neq 0$ implies $\sigma_j^a(\omega)\in E_j$, hence  $\{\partial_jf\neq 0\}\subseteq (\sigma_j^a)^{-1}(E_j)$.

We will now prove that $(\sigma_j^a)^{-1}(E_j)\subseteq\{\partial_jf\neq 0\}$. Assume that $\omega\in (\sigma_j^a)^{-1}(E_j)$. We need to prove that $f(\sigma_j^b(\omega))>f(\sigma_j^a(\omega))$.  Let $\gamma$ be a geodesic on $\sigma_j^b(\omega)$. There are two possibilities: $j\in \gamma$ and $j\not\in \gamma$. In the case $j\in \gamma$, we have
\begin{align*}
f(\sigma_j^b(\omega))&=T(\gamma,\sigma_j^b(\omega))=(b-a)+T(\gamma,\sigma_j^a(\omega))\\
&\geq (b-a)+f(\sigma_j^a(\omega)) >f(\sigma_j^a(\omega)).
\end{align*}
In the case $j\not \in \gamma$, the following holds
\begin{align*}
f(\sigma_j^b(\omega))&= T(\gamma,\sigma_j^b(\omega))=T(\gamma,\sigma_j^a(\omega)).
\end{align*}
However, since $\sigma_j^a(\omega)\in E_j$ and $j\not\in \gamma$, the path $\gamma$ cannot be a geodesic on $\sigma_j^a(\omega)$ and 
$T(\gamma,\sigma_j^a(\omega))>f(\sigma_j^a(\omega))$. We are allowed to conclude that $f(\sigma_j^b(\omega))>f(\sigma_j^a(\omega))$. 
This completes the proof of \eqref{eqn|AiDefinition}.

 We will now prove \eqref{eqn|HatAiDefinition}. Assume first that $\omega\in \left(\sigma_j^b\right)^{-1}(\hat E_j)$. Then, $\sigma_j^b(\omega)\in \hat E_j$, and there is a geodesic $\gamma$ 
on $\sigma_j^b(\omega)$ that passes through $j$.  \begin{align*}
f(\sigma_j^b(\omega))&=T(\gamma,\sigma_j^b(\omega))=T(\gamma,\sigma_j^a(\omega))+(b-a)\\ &\geq f(\sigma_j^a(\omega))+b-a.
\end{align*}
It remains to observe that  
$f(\sigma_j^b(\omega))\leq f(\sigma_j^a(\omega))+ b-a$. Therefore, $\omega\in \hat A_j$. We proved that 
$\left(\sigma_j^b\right)^{-1}(\hat E_j)\subseteq \hat A_j$. 

Assume now that $\omega\in \hat A_j$. From $\hat A_j\subseteq A_j$ and \eqref{eqn|AiDefinition}, we obtain $\sigma_j^a(\omega)\in E_j$. Let $\gamma$ be a geodesic on $\sigma_j^a(\omega)$. Since we assumed that $\omega\in \hat A_j$, we have 
\begin{align*}
f(\sigma_j^b(\omega))&=f(\sigma_j^a(\omega))+(b-a)=T(\gamma,\sigma_j^a(\omega))+(b-a)\\
&=T(\gamma,\sigma_j^b(\omega)).
\end{align*}
This means that $\gamma$ is a geodesic on $\sigma_j^b(\omega)$. Since $\gamma$ passes through $j$, we proved that $\sigma_j^b(\omega)\in \hat E_j$. 
\end{proof}
  
\begin{proposition}\label{thm|forComputer} The following two implications hold for every $\omega\in\Omega$.
\begin{enumerate}
\item[(a)] If $\sigma_j^a(\omega)\in E_j^C$, then $f(\sigma_j^b(\omega))=f(\sigma_j^a(\omega))$;
\item[(b)] If $\sigma_j^b(\omega)\in \hat E_j$, then $f(\sigma_j^b(\omega))=f(\sigma_j^a(\omega))+(b-a)$.
\end{enumerate}
\end{proposition}
\begin{proof}
Part (a) follows directly from \eqref{eqn|AiDefinition}. If we assume $\sigma_j^a(\omega)\in E_j^C$, then \[\omega\in (\sigma_j^a)^{-1}(E_j^C)=
\left((\sigma_j^a)^{-1}(E_j)\right)^C=A_j^C.\] Similarly, part (b) is a direct consequence of \eqref{eqn|HatAiDefinition}. 
\end{proof}

The functions $\sigma_i^a$ and $\sigma_i^b$ are idempotent (a function $\psi:R\to R$ is idempotent if $\psi\circ \psi=\psi$) and they satisfy $\sigma_i^a\circ\sigma_i^b=\sigma_i^a$ and $\sigma_i^b\circ\sigma_i^a=\sigma_i^b$.
These algebraic properties, together with $E_j\subseteq \hat E_j$, $\hat A_j\subseteq A_j$, \eqref{eqn|AiDefinition}, and
 \eqref{eqn|HatAiDefinition} will have the following algebraic consequence: $E_j\subseteq A_j$ and $\hat A_j\subseteq \hat E_j$.
These inclusions (and several more results) will follow from the following general properties of images and pre-images of idempotent functions, whose proofs 
are left for the Appendix.

\begin{proposition}\label{prop|IdempotentFixedPoints} If $\psi:R\to R$ is an idempotent function, then the set of its fixed points is equal to its range, i.e. \begin{align}
\psi(R)&=\left\{x\in R: \psi(x)=x\right\}.
\label{eqn|rangeFixedPoints}
\end{align}\end{proposition}

\begin{proposition}\label{prop|IdempotentPsiXi} Assume that $\psi:R\to R$ and $\xi: R\to R$ are two idempotent functions that satisfy $\psi\circ \xi =\psi$. If $Q$ is any subset of $R$, then 
\begin{align}
&\psi\left(\psi^{-1}(Q)\cap \xi(R)\right)=\psi\left(\psi^{-1}(Q)\right)=Q\cap \psi(R)  =\psi^{-1}(Q)\cap \psi(R). \label{eqn|generalResultIdempotentPsiXi}
\end{align}
\end{proposition}
 
 \begin{proposition}\label{prop|IdempotentPsiAPsiB} Assume that $\psi^a:R\to R$ and $\psi^b: R\to R$ are two idempotent functions that satisfy $\psi^a(R)\cup \psi^b(R)=R$. 
  If  $E$ and $\hat E$ are two subsets of $R$ that  satisfy $E\subseteq \hat E$ and  $(\psi^b)^{-1}(\hat E)\subseteq (\psi^a)^{-1}(E)$, then 
 \begin{align} E  &\subseteq  (\psi^a)^{-1}(E) , \label{eqn|InclusionAEGen}\\ 
  (\psi^b)^{-1}(\hat E)&\subseteq  \hat E.  \label{eqn|InclusionHatAEGeni} 
  \end{align}
 \end{proposition}
 
The equation \eqref{eqn|rangeFixedPoints} applied to $\sigma_j^a$ and $\sigma_j^b$ transforms into \begin{align}
\sigma_j^a\left(\Omega\right)=\{\omega_j=a\}\quad\text{ and }\quad 
\sigma_j^b\left(\Omega\right)=\{\omega_j=b\}.
\label{eqn|fixedPointsSigmaJAB}
\end{align}

We will apply the equation \eqref{eqn|generalResultIdempotentPsiXi}
 to idempotent functions $\sigma_j^a$ and $\sigma_j^b$ that satisfy $\sigma_j^a\circ\sigma_j^b=\sigma_j^a$ and $\sigma_j^b\circ\sigma_j^a=\sigma_j^b$. 

\begin{proposition} For every $j\in W$, the events $E_j$, $\hat E_j$, $A_j$, and $\hat A_j$ satisfy 
\begin{align}
&\sigma_j^a(A_j)=\sigma_j^a(A_j\cap\{\omega_j=b\}) 
=A_j\cap \left\{\omega_j=a\right\} 
=E_j\cap \left\{\omega_j=a\right\}; \label{eqn:AEinclusion}\\
&\sigma_j^b(\hat A_j) = \sigma_j^b(\hat A_j\cap\{\omega_j=a\})  
=\hat A_j\cap \left\{\omega_j=b\right\} 
=\hat E_j\cap \left\{\omega_j=b\right\}. \label{eqn:HatAEinclusion}
\end{align}

\end{proposition}

\begin{proof} The equalities \eqref{eqn:AEinclusion} and \eqref{eqn:HatAEinclusion} are direct consequences of \eqref{eqn|propertyA}, \eqref{eqn|AiDefinition}, \eqref{eqn|HatAiDefinition}, \eqref{eqn|generalResultIdempotentPsiXi},  and \eqref{eqn|fixedPointsSigmaJAB}.
\end{proof}

\begin{theorem} For every $j\in W$, the events $E_j$, $\hat E_j$, $A_j$, and $\hat A_j$ satisfy
\begin{align}
&E_j\subseteq \hat E_j;\quad \hat A_j\subseteq A_j;
\label{eqn:obviousInclusions}\\
&E_j\subseteq A_j ;\quad \text{and}
\label{eqn:EAinclusion}\\
& \hat A_j\subseteq \hat E_j.\label{eqn:HatAEinclusionC}
\end{align}
\end{theorem} 
\begin{proof}
The inclusions \eqref{eqn:obviousInclusions} are obvious, while the inclusions \eqref{eqn:EAinclusion} and \eqref{eqn:HatAEinclusionC} follow from \eqref{eqn|InclusionAEGen} 
and \eqref{eqn|InclusionHatAEGeni}.
\end{proof}

\begin{proposition}\label{thm|IntDependent}
Assume that $0<a<b$ are real numbers. Then
\begin{align}
\left(\exists\,(k_a,k_b)\in\mathbb{Z}^2\right) \quad 
ak_a + bk_b \in (0,\,b-a) 
\quad &\Longleftrightarrow \quad 
\frac{a}{b-a} \notin \mathbb{N}.
\label{eqn|IntDependent}
\end{align} 
\end{proposition}

\begin{proof}
Dividing $ak_a + bk_b \in (0, b-a)$ by $a$, the condition becomes 
$k_a + \frac{b}{a}k_b \in (0, \frac{b}{a}-1)$. Setting $\alpha = \frac{b}{a} - 1 > 0$ and substituting 
$m = k_a + k_b$, $n = k_b$ (a bijection $\mathbb{Z}^2 \to \mathbb{Z}^2$) 
transforms \eqref{eqn|IntDependent} into the equivalent form
\begin{align*} 
\left(\exists\,(m,n)\in\mathbb{Z}^2\right) \quad 
m + n\alpha \in (0,\alpha) 
\quad &\Longleftrightarrow \quad 
\frac{1}{\alpha} \notin \mathbb{N}. 
\end{align*}

\emph{Case 1: $\alpha$ is irrational.} By Dirichlet's approximation 
theorem, for every $\varepsilon > 0$ there exists a positive integer 
$n$ with $n\alpha - \lfloor n\alpha \rfloor \in (0, \varepsilon)$. 
Taking $\varepsilon = \min\{\alpha, 1\}$ and 
$m = -\lfloor n\alpha \rfloor$ gives 
$m + n\alpha = n\alpha - \lfloor n\alpha \rfloor \in (0, \alpha)$. 
Since $\alpha$ is irrational, $\frac{1}{\alpha}$ is also irrational, 
hence not in $\mathbb{N}$.

\emph{Case 2: $\alpha = \frac{p}{q}$ with $\gcd(p,q) = 1$ and $p, q$ 
positive integers.} Multiplying $m + n\alpha \in (0,\alpha)$ by $q$, the condition becomes
\begin{align} 
mq + np & \in (0, p). 
\label{eqn|IntDependentProof02}
\end{align}
Since $\gcd(p,q) = 1$, B\'ezout's identity gives integers $m, n$ 
with $mq + np = 1$, and every integer is obtainable as $mq + np$. 
Hence \eqref{eqn|IntDependentProof02} has a solution if and only 
if the interval $(0, p)$ contains a positive integer, which holds 
if and only if $p > 1$. The condition $p = 1$ is equivalent to 
$\frac{1}{\alpha} = q \in \mathbb{N}$.
\end{proof}

If the real number $a$ is an integer multiple of $b-a$, then the events $\hat A_j$ and $A_j$ coincide, and the relationships between $E_j$, $A_j$, and $\hat E_j$ are simpler; see Figure~\ref{fig|relationshipBetweenAE}(b). Formally, we have the following result.

\begin{proposition}\label{thm|ABSpecial}
Assume that $\frac{a}{b-a}\in\mathbb N$. Then
\[
E_j \subseteq A_j = \hat A_j \subseteq \hat E_j.
\]
\end{proposition}

\begin{proof}
This is an immediate consequence of the inclusions $E_j \subseteq A_j$, $\hat A_j \subseteq \hat E_j$, and the equality $A_j = \hat A_j$.
\end{proof}

\begin{proposition}\label{thm|GeneralVennDiagram} For sufficiently large $n$, the following sets are non-empty:
$\hat E_j\setminus A_j$, $\hat A_j\setminus E_j$,  and $A_j\setminus E_j$.
If $a/(b-a)\notin\mathbb N$, then, for sufficiently large $n$, the sets 
 $A_j\setminus \hat E_j$ and $E_j\setminus \hat A_j$
are non-empty.
\end{proposition}
\begin{proof}
There are trivial examples that establish $\hat E_j\setminus A_j \neq \emptyset$ and $\hat A_j\setminus E_j\neq\emptyset$. We will construct them on the torus model. The examples can be easily extended to the general first-passage percolation. Let $\omega_a$ be the environment that assigns the value $a$ to every edge. It is easy to prove that $\omega_a\in \hat E_j\setminus A_j$, for every edge $j$ that connects the vertex $\overrightarrow x$ with the vertex $\overrightarrow y$ such that $x_1-y_1=\pm 1$. 
 Take now the environment $\omega_b$ whose all edges are assigned the value $b$. Then, every edge between $\overrightarrow x$ and $\overrightarrow y$ for which $x_1-y_1=\pm 1$ 
 is very influential. None of the edges is essential, hence $\omega_b\in \hat A_j\setminus E_j$.
 These two trivial examples $\omega_a$ and $\omega_b$ show that $\hat E_j\setminus A_j$ and $\hat A_j\setminus E_j$ are non-empty. 
 
In Section \ref{sec|relationshipIE}, we will prove that the remaining sets are non-empty. The relation $A_j\setminus E_j\neq\emptyset$ is proved in Proposition \ref{thm|NoOppositeInclusions}.
The Proposition \ref{thm|ABGeneral}
proves that $A_j\setminus \hat E_j$ and $E_j\setminus \hat A_j$ are non-empty if $a/(b-a)\notin \mathbb N$. 
\end{proof}

\begin{proposition} \label{thm|MonotonicityOfGeodesic} Assume that $\omega\in E_j$. A path $\gamma$ is a geodesic on $\omega$ if and only if it is a geodesic on $\sigma_j^a(\omega)$. 
\end{proposition}
\begin{proof} First we will prove that every geodesic on $\omega$ is also a geodesic on $\sigma_j^a(\omega)$. The case $\omega_j=a$ is trivial. Assume that $\omega_j=b$. Let $\mu$ be a geodesic on $\omega$. Since $\omega\in E_j$, we must have \begin{align*}f(\omega)&= T(\mu,\omega)=T(\mu,\sigma_j^a(\omega))+(b-a).\end{align*} We obtained that the equation $T(\mu,\sigma_j^a(\omega))=f(\omega)-(b-a)$ holds for every geodesic $\mu$ on $\omega$. If $\nu$ is any other path that is not a geodesic on $\omega$, we must have 
$T(\nu,\omega)>f(\omega)$. We would also have \begin{align*}T(\nu,\sigma_j^a(\omega))\geq T(\nu,\omega) -(b-a)>f(\omega)-(b-a)=T(\mu,\sigma_j^a(\omega)).\end{align*}
We have made the following conclusion: If $\nu$ is not a geodesic on $\omega$, then $\nu$ is also not a geodesic on $\sigma_j^a(\omega)$. In addition, all geodesics on $\omega$ have the same passage time 
on $\sigma_j^a(\omega)$. This implies that all geodesics on $\omega$ are geodesics on $\sigma_j^a(\omega)$. 

Let us now prove that every geodesic on $\sigma_j^a(\omega)$ is also a geodesic on $\omega$. It suffices to prove this for $\omega\in E_j\cap\{\omega_j=b\}$. Assume the contrary, that there is a geodesic $\gamma$ on $\sigma_j^a(\omega)$ that is not a geodesic on $\omega$. Since \[\sigma_j^a(E_j)\subseteq \sigma_j^a(A_j)=
E_j\cap \{\omega_j=a\}\subseteq E_j,\] we must have $\sigma_j^a(\omega)\in E_j$. Therefore, the path $\gamma$ must pass through $j$ on $\sigma_j^a(\omega)$. Therefore, 
\begin{align}
f(\sigma_j^a(\omega))&=T(\gamma,\sigma_j^a(\omega))=T(\gamma,\omega)-(b-a). \label{eqn|MonotonicityOfGeodesic01}
\end{align}
Since $\gamma$ is not a geodesic on $\omega$, there must be a path $\delta$ which is a geodesic and for which $T(\gamma,\omega)$ is strictly larger than $T(\delta,\omega)$. However, $\omega\in E_j$ by the assumption. Therefore, $\delta$ passes through $j$ and $T(\delta,\omega)=T(\delta,\sigma_j^a(\omega))+(b-a)$. From \eqref{eqn|MonotonicityOfGeodesic01} we obtain 
\begin{align*}
f(\sigma_j^a(\omega))&= T(\gamma,\omega)-(b-a)\\ &> T(\delta,\omega)-(b-a)\\
&= T(\delta,\sigma_j^a(\omega)).
\end{align*}
This is a contradiction, because the value $f$ must be smaller than or equal to the cost over the path $\delta$ on the envifornment $\sigma_j^a(\omega)$. 
\end{proof}

For $V\subseteq W$, define $E_V=\bigcap_{j\in V}E_j$.

\begin{proposition} For every $j\in V$, the following holds 
\begin{align*} 
\sigma_j^a\left(E_V\right)&=E_V\cap \{\omega_j=a\}.  
\end{align*}
\end{proposition}

\begin{proof} The inclusion $\supseteq$ is obvious: If $\omega\in E_V$ and $\omega_j=a$, then $\sigma_j^a(\omega)=\omega$. Therefore, the environment $\omega$ is the image of $\omega$ under $\sigma_j^a$. That makes $\omega$ an element of $\sigma_j^a(E_V)$.

We need to prove that $\sigma_j^a(E_V)\subseteq E_V$. Take $\zeta\in \sigma_j^a(E_V)$. There exists $\omega\in E_V$ such that $\zeta=\sigma_j^a(\omega)$. Since $\omega\in E_V\subseteq E_j$, we can apply Proposition 
\ref{thm|MonotonicityOfGeodesic}. Every geodesic on $\zeta$ must be a geodesic on $\omega$. However, $\omega\in E_V$. Every geodesic on $\omega$ must pass through all of the edges of $V$. All of the geodesics on $\zeta$ must satisfy the same condition. Thus, $\zeta\in E_V$.
\end{proof}

If $\overrightarrow \alpha\in \{a,b\}^m$ and $\overrightarrow v\in W^m$, define $\sigma_{\overrightarrow v}^{\overrightarrow\alpha}:\Omega\to \Omega$ as 
\begin{align*}
\sigma_{\overrightarrow v}^{\overrightarrow\alpha}&=\sigma_{v_1}^{\alpha_1}\circ\cdots\circ\sigma_{v_m}^{\alpha_m},
\end{align*}
where $\alpha_1$, $\dots$, $\alpha_m$ are the components of $\overrightarrow\alpha$ and $v_1$, $\dots$, $v_m$ are the components of $\overrightarrow v$.

The event $I_{\overrightarrow v, \overrightarrow \alpha}$ is defined as 
\begin{align*}
I_{\overrightarrow v,\overrightarrow\alpha}&=\sigma_{\overrightarrow v}^{\overrightarrow \alpha}(\Omega)=\left\{\omega\in\Omega: \omega_{v_1}=\alpha_1, \dots, \omega_{v_m}=\alpha_m\right\}. 
\end{align*}

\begin{proposition}\label{thm|TiltingLemmaMD}
For every event $A$, every vector $\overrightarrow v\in W^k$ that has all distinct components, and every two vectors $\overrightarrow\alpha$ and $\overrightarrow \beta$ from $\{a,b\}^m$, the following holds 
\begin{align}
\mathbb P\left(A\cap I_{\overrightarrow v,\overrightarrow\beta} \right)&= \frac{\mathbb P(\overrightarrow\beta)}{\mathbb P\left(\overrightarrow\alpha\right)} \mathbb P\left( 
\sigma_{\overrightarrow v}^{\overrightarrow \alpha}\left(
A\cap I_{\overrightarrow v,\overrightarrow \beta}
\right)\right).
 \label{eqn|TiltingFormulaMD} 
\end{align}
\end{proposition}
\begin{proof}
Take $\omega\in \Omega$ and fix $\overrightarrow v$. When we remove the components of $\omega$ whose indices appear in the vector $\overrightarrow v$, we obtain a shorter sequence $R_{\overrightarrow v}(\omega)$ that is an element of $\Omega_{\overrightarrow v}=\{a,b\}^{W\setminus\{\overrightarrow v\}}$.  
Let us denote by $\mathbb P_{\overrightarrow v}$ the induced probability measure on $\Omega_{\overrightarrow v}$. 
\begin{align*}
 \mathbb P\left(A\cap I_{\overrightarrow v,\overrightarrow\beta} \right)&=\sum_{\omega \in A\cap I_{\overrightarrow v,\overrightarrow\beta} }\mathbb P(\omega)=
\sum_{\omega \in A\cap I_{\overrightarrow v,\overrightarrow\beta} }\mathbb P(\overrightarrow\beta)\mathbb P_{\overrightarrow v}(R_{\overrightarrow v}(\omega)) \\
&=\frac{\mathbb P(\overrightarrow \beta)}{\mathbb P(\overrightarrow\alpha)}\sum_{\omega \in A\cap I_{\overrightarrow v,\overrightarrow\beta}}\mathbb P(\overrightarrow\alpha)\mathbb P_{\overrightarrow v}(R_{\overrightarrow v}\omega).
\end{align*}
Observe  that 
$\mathbb P(\sigma_{\overrightarrow v}^{\overrightarrow\alpha}(\omega))=\mathbb P(\overrightarrow\alpha)\mathbb P_{\overrightarrow v}(R_{\overrightarrow v}(\omega))$. Therefore, 
\begin{align*} 
\mathbb P\left(A\cap I_{\overrightarrow v,\overrightarrow\beta} \right)
&=\frac{\mathbb P(\overrightarrow \beta)}{\mathbb P(\overrightarrow\alpha)}\sum_{\omega \in A\cap I_{\overrightarrow v,\overrightarrow\beta}}\mathbb P(\sigma_{\overrightarrow v}^{\overrightarrow\alpha}\omega).
\end{align*}
Since 
$\sigma_{\overrightarrow v}^{\overrightarrow \alpha}$ is a bijection from $A\cap I_{\overrightarrow v,\overrightarrow\beta}$ to $\sigma_{\overrightarrow v}^{\overrightarrow\alpha}(A\cap I_{\overrightarrow v,\overrightarrow \beta})$, we can use the substitution $\zeta=\sigma_{\overrightarrow v}^{\overrightarrow \alpha}(\omega)$ in the last summation and obtain the equation \eqref{eqn|TiltingFormulaMD} .
\end{proof}

Let us state a special case of the previous proposition in which the dimensions of vectors are all $1$. 
\begin{proposition}\label{thm|TiltingLemma}
For every event $A$, every pair $(\alpha,\beta)\in\{a,b\}^2$, and every $j\in W$, the following holds 
\begin{align}
\mathbb P\left(A\cap \left\{\omega_j=\beta\right\}\right)&=\frac{\mathbb P(\beta)}{\mathbb P(\alpha)} \mathbb P\left( 
\sigma_j^{\alpha}\left(
A\cap \left\{\omega_j=\beta\right\}\right)\right).
 \label{eqn|TiltingFormula} 
\end{align}
\end{proposition}

We now use \eqref{eqn|TiltingFormula} and \eqref{eqn|TiltingFormulaMD}  to prove the following two theorems.

\begin{theorem}\label{thm|BoundMVGeneral} For every edge $i\in W_n$, the probabilities of the events $A_i$ and $E_i$ satisfy 
\begin{align}
\mathbb P\left(A_i\right)&\leq \frac1{  p}\mathbb P(E_i). \label{eqn|BoundMVGeneral}
\end{align}
\end{theorem}

\begin{proof} Using \eqref{eqn:AEinclusion} we obtain 
\begin{align}
\mathbb P\left(A_i\right)&=\mathbb P\left(A_i\cap \{\omega_i=a\}\right)+\mathbb P\left(A_i\cap \{\omega_i=b\}\right) \nonumber \\
&=\mathbb P\left(E_i\cap \{\omega_i=a\}\right)+\mathbb P\left(A_i\cap \{\omega_i=b\}\right)\nonumber\\
&\leq \mathbb P\left(E_i\right)+\mathbb P\left(A_i\cap \{\omega_i=b\}\right)
. \label{eqn|BoundMV01}
\end{align}
We now use \eqref{eqn|TiltingFormula} and \eqref{eqn:AEinclusion} to bound the second term on the right-hand side of \eqref{eqn|BoundMV01}.
\begin{align}
\mathbb P\left(A_i \cap\left\{\omega_i=b\right\}\right)&=\frac{\mathbb P(b)}{\mathbb P(a)}
\mathbb P\left(\sigma_i^a\left(A_i \cap\left\{\omega_i=b\right\}\right)\right)\nonumber \\ & = \frac{\mathbb P(b)}{\mathbb P(a)}
\mathbb P\left(E_i \cap\left\{\omega_i=a\right\}\right)\leq \frac{\mathbb P(b)}{\mathbb P(a)}
\mathbb P\left(E_i \right). \label{eqn|BoundMV02}
\end{align}
The inequality in \eqref{eqn|BoundMVGeneral} is now a direct consequence of \eqref{eqn|BoundMV01} and \eqref{eqn|BoundMV02}.
\end{proof}

\begin{theorem}\label{thm|multipleFlips}
Assume that $\overrightarrow v\in W^m$ has all distinct components. Assume that $j$ is not a component of $\overrightarrow v$. Assume that $\overrightarrow \gamma$ and $\overrightarrow \delta$ are elements of $\{a,b\}^m$. Then, 
\begin{align}
\mathbb P\left(\left\{\partial_j \left(f\circ \sigma_{\overrightarrow v}^{\overrightarrow\gamma}\right)\neq0\right\} \cap I_{\overrightarrow v,\overrightarrow\delta}\right)
&= 
\frac{\mathbb P(\overrightarrow\delta)}{\mathbb P(\overrightarrow\gamma)} \mathbb P(A_j\cap I_{\overrightarrow v,\overrightarrow\gamma}).
\label{eqn|multipleFlips}
\end{align}
\end{theorem}
\begin{proof} 
By \eqref{eqn|TiltingFormulaMD},
\begin{align*}
&\mathbb P\left(\left\{\partial_j \left(f\circ \sigma_{\overrightarrow v}^{\overrightarrow\gamma}\right)\neq0\right\} \cap I_{\overrightarrow v,\overrightarrow\delta}\right)
\\ &=
\frac{\mathbb P(\overrightarrow\delta)}{\mathbb P(\overrightarrow\gamma)}\mathbb P\left(\sigma_{\overrightarrow v}^{\overrightarrow \gamma}\left(\left\{\partial_j \left(f\circ \sigma_{\overrightarrow v}^{\overrightarrow\gamma}\right)\neq0\right\} \cap I_{\overrightarrow v,\overrightarrow\delta}
\right)
\right).
\end{align*}
It suffices to prove the set equality
\begin{align}
\sigma_{\overrightarrow v}^{\overrightarrow \gamma}\left(\left\{\partial_j \left(f\circ \sigma_{\overrightarrow v}^{\overrightarrow\gamma}\right)\neq0\right\} \cap I_{\overrightarrow v,\overrightarrow\delta}
\right)
&= 
A_j\cap I_{\overrightarrow v,\overrightarrow\gamma}. \label{eqn|multileFlips01}
\end{align}
Since $j$ is not a component of $\overrightarrow v$, repeated application of the property \eqref{eqn|propertyB} gives 
$\sigma_{\overrightarrow v}^{\overrightarrow \gamma}\circ \sigma_j^a=\sigma_j^a\circ \sigma_{\overrightarrow v}^{\overrightarrow \gamma}$ and 
$\sigma_{\overrightarrow v}^{\overrightarrow \gamma}\circ \sigma_j^b=\sigma_j^b\circ \sigma_{\overrightarrow v}^{\overrightarrow \gamma}$. Hence
\begin{align*}
\partial_j 
\left(f\circ \sigma_{\overrightarrow v}^{\overrightarrow\gamma}\right)
&=
f\circ\sigma_j^b\circ\sigma_{\overrightarrow v}^{\overrightarrow\gamma}-f\circ\sigma_j^a\circ\sigma_{\overrightarrow v}^{\overrightarrow\gamma}
=
(\partial_j f)\circ \sigma_{\overrightarrow v}^{\overrightarrow\gamma},
\end{align*}
so
\begin{align*}
\left\{\partial_j \left(f\circ \sigma_{\overrightarrow v}^{\overrightarrow\gamma}\right)\neq 0\right\}
&=
\left\{(\partial_j f)\circ \sigma_{\overrightarrow v}^{\overrightarrow\gamma}\neq 0\right\}
=
\left(\sigma_{\overrightarrow v}^{\overrightarrow\gamma}\right)^{-1}(A_j).
\end{align*}
We now use the standard set-theoretic identity 
$g\!\left(g^{-1}(C)\cap D\right)=C\cap g(D)$, valid for every function 
$g:X\to Y$ and all subsets $D\subseteq X$, $C\subseteq Y$. Combined 
with $\sigma_{\overrightarrow v}^{\overrightarrow \gamma}
\!\left(I_{\overrightarrow v,\overrightarrow\delta}\right)
=I_{\overrightarrow v,\overrightarrow\gamma}$, this yields
\begin{align*}
\sigma_{\overrightarrow v}^{\overrightarrow \gamma}\!\left(
\left(\sigma_{\overrightarrow v}^{\overrightarrow\gamma}\right)^{-1}\!(A_j)
\cap I_{\overrightarrow v,\overrightarrow\delta}\right)
&=
A_j\cap \sigma_{\overrightarrow v}^{\overrightarrow \gamma}\!\left(
I_{\overrightarrow v,\overrightarrow\delta}\right)
=
A_j\cap I_{\overrightarrow v,\overrightarrow\gamma},
\end{align*}
which proves \eqref{eqn|multileFlips01}.
\end{proof}

\section{Variance decomposition}

\subsection{Integration by parts} We will derive several theorems that hold for general random variables that are related to Boolean functions. The standard theory can be found in \cite{ODonnell2014}. We are deviating in minor ways and we will end up with somewhat different formulas. We have defined the derivative $\partial_i$ in a non-standard way to make our future results about $A_i$, $E_i$, $\hat A_i$, and $\hat E_i$ the simplest possible. We will also deviate a little from the literature when we define the Fourier basis below. Due to these changes, and since the results we need are elementary, it is most convenient to re-derive them from scratch.

Theorems \ref{thm|integrationByPartsSpec} and \ref{thm|integrationByParts} are analogous to integration by parts formulas from calculus.

\begin{proposition}\label{thm|BayesFormula} For every random variable $\varphi$ on $\Omega$, its expected value $\mathbb E[\varphi]$ can be evaluated using 
the equation
 \begin{align}
\mathbb E\left[\varphi\right]&= p\mathbb E[\varphi\circ \sigma_i^a]+(1-p)\mathbb E[\varphi\circ \sigma_i^b]. 
\label{eqn|BayesFormula}
\end{align}
\end{proposition}
\begin{proof}
Let us denote by $\mathbb P_i$ the product measure on $\{a,b\}^{W\setminus \{i\}}$ defined as 
\begin{align*} 
\mathbb P_i(\omega)&= \prod_{k\in W\setminus\{i\}} p^{\mathbf1_a(\omega_k)}(1-p)^{1-\mathbf1_a(\omega_k)}. 
\end{align*}  Let $\mathbb  E_i$ be the corresponding expected value. We use \eqref{eqn|propertyF} to obtain  
\begin{align} \nonumber
\mathbb E[\varphi]&= \mathbb E\left[\varphi\cdot 1_{\omega_i=a}\right]+\mathbb E\left[\varphi\cdot 1_{\omega_i=b}\right]\\
\nonumber 
&= \mathbb E\left[\varphi\circ \sigma_i^a\cdot 1_{\omega_i=a}\right]+\mathbb E\left[\varphi\circ \sigma_i^b\cdot 1_{\omega_i=b}\right]\\
&=  p\mathbb E_i\left[\varphi\circ \sigma_i^a\right]+(1-p)\mathbb E_i\left[\varphi\circ \sigma_i^b\right]. \label{eqn|BayesFormulaIntermediateResult}
\end{align}
We can apply \eqref{eqn|BayesFormulaIntermediateResult} to $\varphi\circ \sigma_i^a$ and use \eqref{eqn|propertyA}.  
\begin{align*}
\mathbb E[\varphi\circ \sigma_i^a]
&= p\mathbb E_i\left[\varphi\circ \sigma_i^a\circ\sigma_i^a\right]+(1-p)\mathbb E_i\left[\varphi\circ \sigma_i^a\circ \sigma_i^b\right]\\
&= p\mathbb E_i\left[\varphi\circ \sigma_i^a\right]+(1-p)\mathbb E_i\left[\varphi\circ \sigma_i^a\right]\\
&=\mathbb E_i\left[\varphi\circ\sigma_i^a\right].
\end{align*}
 In an analogous way we derive the equality $\mathbb E\left[\varphi\circ \sigma_i^b\right]=\mathbb E_i\left[\varphi \circ\sigma_i^b\right]$. The equation \eqref{eqn|BayesFormulaIntermediateResult} becomes  \eqref{eqn|BayesFormula}. 
\end{proof}

For $\omega\in \Omega$ and $i\in W$, let us define 
\begin{align*}
r_i(\omega)&=\left\{\begin{array}{ll} 
-\sqrt{\frac{1-p}{p}},& \text{if } \omega_i=a,\\
\sqrt{\frac{p}{1-p}},& \text{if } \omega_i=b.
\end{array}
\right.   
\end{align*}
For each $i$, we have $\mathbb E[r_i]=0$ and $\text{var}(r_i)=1$.
For $S\subseteq W$, we define

 \begin{align*}r_S(\omega)&=\prod_{i\in S}r_i(\omega). 
\end{align*}

Due to independence, we have $\mathbb E\left[r_S\right]=0$ and $\text{var}(r_S)=1$ whenever $S\neq \emptyset$. Also, if $S$ and $T$ are different sets, then their dot product 
$\mathbb E\left[r_Sr_T\right]$ must be equal to $0$. 
Therefore, the functions $\left(r_S\right)_{S\subseteq W}$ form an orthonormal basis of $L^2(\Omega)$; see Section 2 of \cite{talagrand1994}.

\begin{theorem}\label{thm|integrationByPartsSpec} 
For every nonempty $S\subseteq W$ and every random variable $\varphi$, the expected value of $\varphi r_S$ satisfies
\begin{align*}
\mathbb E\left[\varphi r_S\right]&= \sqrt{p(1-p)}^{|S|}\mathbb E\left[ \partial_S\varphi  \right]. 
\end{align*}
\end{theorem}

 Theorem \ref{thm|integrationByPartsSpec} is a special case obtained by placing $S=T$ in the following more general result.  

\begin{theorem}\label{thm|integrationByParts}
For every two sets $T$ and $S$ with $T\subseteq S\subseteq W$, and every random variable $\varphi$, the expected value of $\varphi r_S$ satisfies \begin{align}
\mathbb E\left[\varphi r_S\right]&= \sqrt{p(1-p)}^{|T|}\mathbb E\left[\left(\partial_T\varphi\right) \cdot r_{S\setminus T}\right].
\label{eqn|integrationByParts}
\end{align} 
\end{theorem}

\begin{proof}
Let us first consider the case $T=\{i\}$. An application of \eqref{eqn|BayesFormula} results in 
\begin{align*}
\mathbb E\left[\varphi r_S\right]&=p\mathbb E\left[\varphi r_S\circ\sigma_i^a\right]+(1-p)\mathbb E\left[\varphi r_S\circ\sigma_i^b\right]\\
&=p\mathbb E\left[\left(\varphi\circ\sigma_i^a\right) \left(r_S\circ\sigma_i^a\right)\right]+(1-p)\mathbb E\left[\left(\varphi \circ\sigma_i^b\right)\left( r_S\circ\sigma_i^b\right)\right]\\
&=-p\mathbb E\left[\left(\varphi\circ\sigma_i^a\right) \cdot \sqrt{\frac{1-p}{p}} r_{S\setminus \{i\}} \right] \\
&\quad+(1-p)\mathbb E\left[\left(\varphi \circ\sigma_i^b\right)
\cdot  \sqrt{\frac{p}{1-p}} r_{S\setminus \{i\}}\right]\\
&=\sqrt{p(1-p)}\mathbb E\left[\left(\varphi\circ\sigma_i^b-\varphi\circ\sigma_i^a\right)  r_{S\setminus \{i\}} \right]\\
&=\sqrt{p(1-p)}\mathbb E\left[\partial_i\varphi r_{S\setminus\{i\}}\right].
\end{align*}
The general result (\ref{eqn|integrationByParts}) follows by a straightforward induction. 
\end{proof}

\begin{proof}[Proof of Theorem \ref{thm|varianceFormula}] The coefficients of $f$ in base $\left(r_S\right)_{S\subseteq W}$ will be denoted by $a_S$. The coefficients satisfy 
\begin{align*}
a_S&=\mathbb E\left[fr_S\right]=\sqrt{p(1-p)}^{|S|}\mathbb E\left[\partial_Sf\right].
\end{align*}
Except for the coefficient $a_{\emptyset}$, the random variables $f-\mathbb E[f]$ and $f$ have the same coefficients. This implies \eqref{eqn|varianceFormula}.
\end{proof}
 
\subsection{Beckner-Bonami inequality}
We will use the following variant of Beckner-Bonami inequality. 
\begin{theorem}\label{thm|Beckner} 
If $P_Lf$ is the projection of $f$ to the $\text{span} \left\{r_Q: |Q|\leq L\right\}$, then for each $q\geq 2$, there exists a constant $\alpha=\alpha(p,q)>0$ such that
\begin{align*}\|P_Lf\|_2^2&\leq e^{\alpha L}\|f\|_{q'}^2,
\end{align*}
where $q'$ is the conjugate of $q$.
\end{theorem}
\begin{proof} This is an easy consequence of  a result proved in \cite{talagrand1994}. Proposition 2.2 on page 1580 of \cite{talagrand1994} states the following:
If $a_S=\mathbb E[fr_S]$, then for fixed $k$ the following holds
\begin{align}
\sum_{|S|=k} a_S^2 &\leq (q-1)^k\theta^{2k}\cdot \|f\|_{q'}^2, \label{eqn|BecknerBonami}
\end{align}
where $\theta=\frac1{\sqrt{p(1-p)}}$.
Substituting $\beta=\frac{q-1}{ p(1-p) }$ and summing \eqref{eqn|BecknerBonami} for $k\in\{1,2,\dots, L\}$ gives us 
\begin{align*}
\|P_Lf\|_2^2&\leq \frac{\beta^{L+1}-1}{\beta-1}\cdot \|f\|_{q'}^2.
\end{align*}
Clearly, we can find $\alpha$ that depends only on $p$ and $q$ such that $\frac{\beta^{L+1}-1}{\beta-1}\leq e^{\alpha L}$. 
\end{proof}

\subsection{Bound on the remainder $R_k$}\label{subs|ProofRk} We now prove Theorem \ref{thm|BoundRemainder} that was stated in the introduction. 
\begin{proof}[Proof of Theorem \ref{thm|BoundRemainder}]   
 Let $\mathcal F$ be the family of those subsets of $W$ that have cardinality at least $k$.  
Let us introduce some notation that will simplify the writing. We will denote by $\mathcal I_T$ the set of all subsets of $W$ that contain $T$, i.e. 
\begin{align*}
\mathcal I_T&=\left\{S\subseteq W: T\subseteq S\right\}.
\end{align*}
We will often write $\mathcal I_t$ instead of $\mathcal I_{\{t\}}$. 
\begin{align*}
R_k(f)&=\sum_{S\in\mathcal F} a_S^2= \sum_{S\in\mathcal F}a_S^2\cdot\frac1{|S|}\cdot \sum_{i_1\in W} 1_{i_1\in S}=
\sum_{i_1\in W}\left(\sum_{S\in\mathcal F\cap\mathcal I_{i_1}}\frac{a_S^2}{|S|}\right).
\end{align*}
The right-hand side can be further expanded using indicators.
\begin{align*}
R_k(f)
&=\sum_{i_1\in W}\left(\sum_{S\in\mathcal F\cap\mathcal I_{i_1}}\frac{a_S^2}{|S|(|S|-1)}\sum_{i_2\in W\setminus\{i_1\}} 1_{i_2\in S}\right).\end{align*}
We can change the order of summations and obtain 
\begin{align}
R_k(f)
&=\sum_{i_1\in W}\left(\sum_{i_2\in W\setminus\{i_1\}}\left(\sum_{S\in\mathcal F\cap \mathcal I_{\{i_1,i_2\}}}\frac{a_S^2}{|S|(|S|-1)}\right)\right).
\label{eqn|VBoundL01}
 \end{align}
 We can now continue in the same way as in \eqref{eqn|VBoundL01} until the number of indices becomes $k$. 
 The summation \eqref{eqn|VBoundL01} becomes 
 \begin{align} 
R_k(f)
 &=\sum_{M\subseteq W, |M|=k}k!\cdot \left(\sum_{S\in\mathcal F\cap\mathcal I_M}\frac{a_S^2}{|S|(|S|-1)(|S|-2)\cdots (|S|-k+1)}\right). \label{eqn|VBoundL01P5}
\end{align}
For each $S$ with $M\subseteq S$, we apply \eqref{eqn|integrationByParts} with $\varphi=f$ and $T=M$ to obtain
\begin{align} a_S&=\mathbb E\left[f\cdot r_S\right]= \sqrt{p(1-p)}^{|M|}\cdot\mathbb E\left[\partial_Mf\cdot r_{S\setminus M}\right]. 
\label{eqn|VBoundL01P7}
\end{align}
Equations \eqref{eqn|VBoundL01P5} and \eqref{eqn|VBoundL01P7} imply
\begin{align}
R_k(f)
 &= \sum_{M\subseteq W, |M|=k}k!\cdot\left(\sum_{S\in\mathcal F\cap\mathcal I_M}\frac{(p(1-p))^{k}\mathbb E\left[\left(\partial_Mf\right)r_{S\setminus M}\right]^2}{|S|(|S|-1)(|S|-2)\cdots (|S|-k+1)}\right).
 \label{eqn|VBoundL02}
 \end{align}
For fixed $M$, let us denote by $\Sigma(M)$ the inner summation in \eqref{eqn|VBoundL02}. Formally, 
\begin{align}
\Sigma(M)&=k!\cdot\left(\sum_{S\in\mathcal F\cap\mathcal I_M}\frac{(p(1-p))^{k}\mathbb E\left[\left(\partial_Mf\right)r_{S\setminus M}\right]^2}{|S|(|S|-1)(|S|-2)\cdots (|S|-k+1)}\right). 
\label{eqn|DefSigmaM}
\end{align}
We will assume that the summation is restricted to the sets $M$ for which $\|\partial_Mf \|_1$ is non-zero. Since we are working with finite sample space
in which every outcome has positive probability, the $L^1$-norm is zero only when the function $\partial_Mf$ is identically equal to $0$. 
We split $\Sigma(M)$ into two groups: the summation $H_M^-$ corresponding to sets of sizes smaller than $L_M$; and the summation $H_M^+$ corresponding to sets of sizes larger than or equal to $L_M$. The integer $L_M$ will be determined later. If the number of elements of $S$ is at least $k$, then the product of the numbers 
 $|S|$, $|S|-1$, $\dots$, $|S|-k+1$ in the denominator  is greater than or equal to $k!$. Hence,
\begin{align*}
 H_M^- &= k!\cdot  \sum_{S\in\mathcal I_M, k\leq|S|<L_M}\frac{(p(1-p))^{k}\mathbb E\left[\left(\partial_Mf\right)r_{S\setminus M}\right]^2}{|S|(|S|-1)(|S|-2)\cdots (|S|-k+1)} \\
&\leq  {(p(1-p))^k} \left(\sum_{S\in\mathcal I_M, k\leq|S|<L_M}
\mathbb E\left[\left(\partial_Mf\right)r_{S\setminus M}\right]^2\right). \end{align*}
There is an obvious bijection between $\mathcal I_M$ and the subsets of $W\setminus M$. Hence, we can do the substitution $Q=S\setminus M$ and obtain the following bound for $H_M^-$.
\begin{align*}  H_M^-  &\leq {(p(1-p))^k}  \sum_{Q\subseteq W\setminus M, |Q|< L_M-k} 
\mathbb E\left[\left(\partial_Mf\right)r_{Q}\right]^2 .
\end{align*}
We will now use \eqref{eqn|VBoundL02} to find an upper bound for $H_M^+$. The cardinalities of sets $S$ are now bigger than or equal to $L_M$. 
Therefore, the product of numbers  $|S|$, $|S|-1$, $\dots$, $|S|-k+1$  in the denominator is greater than or equal to $L_M\cdot (L_M-1)\cdots (L_M-k+1)$ which is equal to $k!\cdot \binom{L_M}{k}$. 
Therefore, the sum $H_M^+$ satisfies
\begin{align*}
\nonumber H_M^+&=k!  \sum_{S\in\mathcal I_M, |S|\geq L_M}\frac{(p(1-p))^{k}\mathbb E\left[\left(\partial_Mf\right)r_{S\setminus M}\right]^2}{|S|(|S|-1)(|S|-2)\cdots (|S|-k+1)}\\
\nonumber
&\leq\frac{(p(1-p))^k}{\binom{L_M}{k}} \sum_{S\in\mathcal I_M, |S|\geq L_M}
\mathbb E\left[\left(\partial_Mf\right)r_{S\setminus M}\right]^2 \\
\nonumber
&=\frac{(p(1-p))^k}{\binom{L_M}{k}}\sum_{Q\subseteq W\setminus M, |Q|\geq L_M-k} 
\mathbb E\left[\left(\partial_Mf\right)r_{Q}\right]^2.
\end{align*}
Let us now set our first requirement for $L_M$. This requirement will be $L_M\geq 2k$. Then, each of the numbers $L_M$, $L_M-1$, $\dots$, $L_M-k+1$ is greater than or equal to $\frac{L_M}2$ and their product is at 
least $L_M^k/2^k$. Therefore,  $\binom{L_M}{k}\geq \frac{L_M^k}{2^k\cdot k!}$. The expansion \eqref{eqn|VBoundL02} allows us to bound $R_k(f)$ as follows.
\begin{align}\nonumber
R_k(f)&\leq
(p(1-p))^k\sum_{M\subseteq W, |M|=k} \left(
 \sum_{Q\subseteq W\setminus M, |Q|<L_M-k}\mathbb E\left[\left(\partial_Mf\right)r_Q\right]^2\right.\\
  &\quad\left.
+
\frac{2^kk!}{L_M^k}\sum_{Q\subseteq W\setminus M, |Q|\geq L_M-k}\mathbb E\left[\left(\partial_Mf\right)r_Q\right]^2
\right). \label{eqn|VBoundL03Before}
\end{align}
The projection $P_{L_M-k}(\partial_Mf)$ of $\partial_Mf$ to the subspace spanned by $\{r_Q\}$ for sets $Q$ of cardinality strictly smaller than $L_M-k$ satisfies 
\begin{align*}
P_{L_M-k}(\partial_Mf)&=\sum_{Q\subseteq W\setminus M,|Q|<L_M-k} \mathbb E\left[\left(\partial_Mf\right)r_Q\right]r_Q.
\end{align*}
Therefore, the inequality \eqref{eqn|VBoundL03Before} becomes 
\begin{align}\nonumber 
R_k(f)&\leq (p(1-p))^k\cdot\sum_{M\subseteq W, |M|=k} \left(
  \left\|P_{L_M-k}(\partial_Mf)\right\|_2^2\right.\\
  &\quad\left.
+
\frac{2^kk!}{L_M^k}\sum_{Q\subseteq W\setminus M, |Q|\geq L_M-k}\mathbb E\left[\left(\partial_Mf\right)r_Q\right]^2
\right). \label{eqn|VBoundL03Almost}
\end{align}
The second term on the right-hand side of \eqref{eqn|VBoundL03Almost} has an excellent coefficient $L_M^k$ in the denominator. We can afford the following generous bound 
\[\sum_{Q\subseteq W\setminus M, |Q|\geq L_M-k}\mathbb E\left[\left(\partial_Mf\right)r_Q\right]^2
 \leq \left\|\partial_Mf\right\|_2^2.\]
The inequality \eqref{eqn|VBoundL03Almost} becomes 
\begin{align} 
R_k(f) 
&\leq
(p(1-p))^k\sum_{M\subseteq W, |M|=k} \left(
  \left\|P_{L_M-k}(\partial_Mf)\right\|_2^2
+
\frac{2^kk!}{L_M^k}\left\|\partial_Mf\right\|_2^2
\right). \label{eqn|VBoundL03}
\end{align}
Since the cardinalities are strictly smaller than $L_M$, we can use Theorem \ref{thm|Beckner}.
We will take $q'=\frac32$. 
There exists a scalar $\alpha$ such that 
\begin{align}
  \left\|P_{L_M-k}(\partial_Mf)\right\|_2^2&\leq  e^{\alpha L_M}\|\partial_Mf\|_{3/2}^2.  \label{eqn|Beckner}
\end{align}
We rewrite the right-hand side of \eqref{eqn|Beckner} as $\|\partial_Mf\|_{3/2}^2 = \mathbb E\left[|\partial_Mf|^{3/2}\right]^{4/3}$. Applying the Cauchy-Schwarz inequality to $|\partial_Mf|$ and $|\partial_Mf|^{1/2}$,
\begin{align}
\mathbb E\left[|\partial_Mf|^{3/2}\right] &= \mathbb E\left[|\partial_Mf|\cdot |\partial_Mf|^{1/2}\right] \nonumber \\
&\leq \mathbb E\left[|\partial_Mf|^2\right]^{1/2}\cdot \mathbb E\left[|\partial_Mf|\right]^{1/2} \nonumber \\
&= \|\partial_Mf\|_2\cdot \|\partial_Mf\|_1^{1/2}. \label{eqn|CauchySchwarzPartialM}
\end{align}
Raising \eqref{eqn|CauchySchwarzPartialM} to the $\frac43$ power and combining with \eqref{eqn|Beckner},
\begin{align}
\|P_{L_M-k}(\partial_Mf)\|_2^2 &\leq  e^{\alpha L_M}\|\partial_Mf\|_2^{4/3}\cdot \|\partial_Mf\|_1^{2/3} \nonumber \\
&= \frac{e^{\alpha L_M}}{\left(\|\partial_Mf\|_2/\|\partial_Mf\|_1\right)^{2/3}}
\|\partial_Mf\|_2^2 . \label{eqn|VBoundL05}
\end{align}
Using \eqref{eqn|VBoundL05} we transform \eqref{eqn|VBoundL03} into 
\begin{align} 
R_k(f)&\leq
\left(p(1-p)\right)^k \nonumber 
\\
&\quad\times\sum_{M\subseteq W, |M|=k}
\left(\frac{e^{\alpha L_M}}{\left(\|\partial_Mf\|_2/\|\partial_Mf\|_1\right)^{2/3}}
+\frac{2^kk!}{L_M^k}\right)
\|\partial_Mf\|_2^2.
\label{eqn|VBoundL07Unfinished}
\end{align}
Let us introduce $ \theta=\left(p(1-p)\right)^k\cdot 2^k\cdot k!$ and 
\begin{align} B_M&=\left(\frac{\|\partial_Mf\|_2}{\|\partial_Mf\|_1}\right)^{\frac23}.
\label{eqn|DefinitionBMN}\end{align}
 Define the function $\psi_M$ as 
\begin{align} \psi_M(L) &=\frac{e^{\alpha L}}{B_M}+\frac1{L^k},
\label{eqn|definitionOfPsi}\end{align} we can re-write \eqref{eqn|VBoundL07Unfinished} as 
\begin{align} 
R_k(f)&\leq
\theta \sum_{M\subseteq W, |M|=k}\psi_M(L_M)\cdot\|\partial_Mf\|_2^2.
\label{eqn|VBoundL07}
\end{align} 
We now choose $L_M$, subject to $L_M\geq 2k$.
Observe that if $B_M>e^{6\alpha k}$, then $\left\lfloor\frac{\log B_M}{2\alpha}\right\rfloor\geq 2k$. Let $B_0$ be the real number such that $B>B_0$ implies $B^{\frac1{2k}}>\log B$. Let $\hat B=\max\{e^{6\alpha k},B_0\}$.
 
If we assume that $B_M>\hat B$, then we choose \[L_M=\left\lfloor\frac{\log B_M}{2\alpha}\right\rfloor.\]
This choice satisfies $L_M\geq 2k$, as required, and immediately implies $L_M>\frac1{3\alpha}\log B_M$, hence \begin{align}\frac{1}{L_M^k} &< \frac{\left(3\alpha\right)^k}{\log^kB_M}. \label{eqn|PsiBound1}\end{align}

The inequality $L_M<\frac1{2\alpha}\log B_M$ gives us 
\begin{align} \frac{e^{\alpha L_M}}{B_M}&<\frac{e^{\frac12\log B_M}}{B_M}=\frac1{\sqrt{B_M}}<\frac1{\log^kB_M}.
\label{eqn|PsiBound2}
\end{align}

Inequalities \eqref{eqn|PsiBound1} and \eqref{eqn|PsiBound2} imply that if $B_M>\hat B$, then
\begin{align}
\psi_M(L_M)&<\frac{1+(3\alpha)^k}{\log^kB_M}. \label{eqn|PsiBoundFinal}
\end{align}
If $B_M\leq \hat B$, then we are going to use a much simpler bound for $\Sigma(M)$ defined in \eqref{eqn|DefSigmaM}. The product of numbers $|S|$, $(|S|-1)$, $\dots$, $(|S|-k+1)$ in the denominator is at least as big as $k!$. This generous bound is sufficient to cancel $k!$ and we are left with $\Sigma(M)\leq (p(1-p))^k\|\partial_Mf\|_2^2$. 
However, since $B_M\leq \hat B$, we have 
\begin{align}
\Sigma(M)&\leq (p(1-p))^k\|\partial_Mf\|_2^2\cdot \frac{1+\log^k \hat B}{1+\log^kB_M}. \label{eqn|TrivialBoundForSmallBM}
\end{align}
Let us now use \eqref{eqn|DefinitionBMN} to replace $\log B_M$ with $\frac23\log \frac{\|\partial_Mf\|_2}{\|\partial_Mf\|_1}$. In the case $B_M\leq \hat B$ we apply \eqref{eqn|TrivialBoundForSmallBM}, while in the case $B_M>\hat B$ we apply \eqref{eqn|VBoundL07} and \eqref{eqn|PsiBoundFinal} to conclude that there exists a constant $C_k\in\mathbb R$ 
such that 
\begin{align*}
R_k(f)&\leq C_k\cdot \sum_{M\subseteq W,|M|=k}\frac{\|\partial_Mf\|_2^2}{1+\left(\log \frac{\|\partial_Mf\|_2}{\|\partial_Mf\|_1}\right)^k}.  
\end{align*}
This completes the proof of Theorem \ref{thm|BoundRemainder}. 
\end{proof}

\noindent{\em Remark.} The function $\psi_M(L)$ defined in \eqref{eqn|definitionOfPsi} cannot be bounded by something much better than $\log^{-k}B_M$, as was done in \eqref{eqn|PsiBoundFinal}. Basic analysis of $\psi_M$ shows that it is convex and increasing for positive $L$. Its minimum is attained at the solution of the equation $\psi'(L)=0$, which after the substitutions \[y=\frac{\alpha}{k+1}L\quad\text{ and }\quad x=\frac{\alpha}{k+1}\left(\frac{kB_M}{\alpha}\right)^{1/(k+1)}\] becomes 
$ye^y=x$. The function $x\mapsto y(x)$ is not an elementary function, but it is very easy to prove that it is increasing and 
\[\lim_{x\to+\infty}\frac{y(x)}{\log x}=1.\]

\subsection{First-passage percolation on torus}
We now turn to first-passage percolation time $f^{\tau}$ on torus.

\begin{proposition}\label{thm|BoundMVTorus} In the torus model, for every $i\in W_n$ we have 
\begin{align}
\mathbb P\left(A_i\right)&\leq \frac1{  p}\mathbb P(E_i)\leq\frac{b}{apn^{d-1}}.   \label{eqn|BoundMVTorus}
\end{align}
\end{proposition}

\begin{proof} The first inequality follows from \eqref{eqn|BoundMVGeneral}. 
Due to symmetry, the values $\mathbb P(E_j)$ are equal among edges $j$ that are parallel to the first coordinate vector. Also, the values $\mathbb P(E_j)$ are equal among edges $j$ that are orthogonal to the first coordinate vector. Neither of these values are $0$ because for every fixed edge $j$, we can easily construct an environment on which the edge $j$ is essential. Hence, $\mathbb P(E_j)$ could take two possible values. Let us call them $P_1$ and $P_2$ and assume that $P_1<P_2$. Our entire graph is $[-2n,2n]^d$. There is a total of $2^dn^d$ edges parallel to the first coordinate vector. There is a total of $(d-1)\cdot 2^d\cdot n^d$ edges orthogonal to the first coordinate vector. Hence, there are more than $n^d$ edges $j$ for which $\mathbb P(E_j)$ is the larger number $P_2$.
The total sum of all of $\mathbb P(E_j)$ is greater than $n^d\cdot P_2$. Hence, for every edge $i$, we must have 
\begin{align*}
\mathbb P(E_i)&< \frac1{n^d}\sum_{j\in W}\mathbb P(E_j)=\frac1{n^d}\mathbb E\left[\sum_{j\in W}1_{E_j}\right].
\end{align*}
The sum $\sum_{j\in W}1_{E_j}$ is bounded by $bn/a$ because all of the essential edges must be on one geodesic whose length is at most $\frac{bn}a$.
\end{proof}

\begin{theorem}\label{thm|PartialNonZero} For every integer $k\geq 1$, there exist a constant $\theta=\theta(k,p)$ and an integer $N_0$ such that for every $n\geq N_0$ and every $M\subseteq W$ with $|M|=k$, the first passage percolation time $f^{\tau}$ on torus satisfies
\begin{align}
\mathbb P\left(\partial_Mf^{\tau}\neq 0\right) &\leq \frac{\theta}{n^{d-1}}.
\label{eqn|PartialNonZero}
\end{align}
\end{theorem}
\begin{proof} We will omit the superscript $\tau$. However, the argument in this proof applies only to the first passage percolation on torus.  
The inequality is obvious if $\partial_Mf=0$ almost surely. Assume that $\partial_Mf(\omega)\neq 0$ for some $\omega$. 
There is an ordering $(m_1, \dots, m_k)$ of the set $M$ and a vector $\overrightarrow\alpha=(\alpha_2,\dots, \alpha_k)\in \{a,b\}^{k-1}$ such that 
 $\partial_{m_1} f(\sigma_{m_2}^{\alpha_2}\circ\cdots\circ \sigma_{m_k}^{\alpha_k}(\omega))\neq 0$. Let us denote $\overrightarrow v=(m_2,\dots, m_k)$. We will now sum over all possible elements $m_1$ and 
 all possible choices of $\overrightarrow v$ and $\overrightarrow \alpha$. 
\begin{align}
\nonumber 
\mathbb P\left(\partial_Mf\neq 0\right)&\leq \sum_{m_1,\overrightarrow v,\overrightarrow \alpha}
\mathbb P\left(\partial_{m_1}f\circ \sigma_{\overrightarrow v}^{\overrightarrow\alpha}\neq 0\right)\\ 
&= 
\sum_{m_1,\overrightarrow v,\overrightarrow \alpha}\sum_{\overrightarrow\beta\in\{a,b\}^{k-1}}
\mathbb P\left(\left\{\partial_{m_1}f\circ \sigma_{\overrightarrow v}^{\overrightarrow\alpha}\neq 0\right\}\cap I_{\overrightarrow v,\overrightarrow \beta}\right) . \label{eqn|ManySummations}
\end{align}
We now use \eqref{eqn|multipleFlips} to obtain 
\begin{align*}\mathbb P\left(\left\{\partial_{m_1}f\circ \sigma_{\overrightarrow v}^{\overrightarrow\alpha}\neq 0\right\}\cap I_{\overrightarrow v,\overrightarrow \beta}\right)
&\leq \frac{\mathbb P\left(\overrightarrow \beta\right)}{\mathbb P\left(\overrightarrow \alpha\right)} \mathbb P\left(A_{m_1}\right)\\
&\leq\left(\frac{\max\{p,1-p\}}{\min\{p,1-p\}}\right)^{k-1}\cdot\frac{b}{apn^{d-1}},
\end{align*}
where for the last inequality we used \eqref{eqn|BoundMVTorus}. The number of terms in the summations 
 \eqref{eqn|ManySummations} is really large and is exponential in $k$, because we are summing over all possibilities for $m_1$, $\overrightarrow v$, $\overrightarrow \alpha$, and $\overrightarrow \beta$. 
 However, the number of terms depends only on $k$. Therefore, the last inequality and \eqref{eqn|ManySummations} can be used to conclude that 
 there exists a scalar $\theta$ for which \eqref{eqn|PartialNonZero} is satisfied.
\end{proof}

\subsection{Bound on the lower sum $L_k$}\label{subs|ProofLk} 
\begin{proof}[Proof of Theorem \ref{thm|SumLowerFourierLevels}]
The bound is trivial for $k=1$. Let us assume $k\geq 2$. For every fixed $m\geq 1$, the number $\Sigma_m(f^{\tau})$ satisfies
\begin{align}
\Sigma_m(f^{\tau})&=\sum_i \sum_{S\ni i, |S|=m}\frac{\left(\mathbb E\left[\partial_Sf^{\tau}\right]\right)^2}{|S|} \nonumber\\
&\leq \sum_i \sum_{S\ni i, |S|=m} \left(\mathbb E\left[\partial_Sf^{\tau}\right]\right)^2
=\sum_i \sum_{T:i\not \in T, |T|=m-1} \left(\mathbb E\left[\partial_T\partial_i f^{\tau}\right]\right)^2,
\label{eqn|beforeBecknerBonami}
\end{align}
We will convert the inner summation on the right-hand side of \eqref{eqn|beforeBecknerBonami} to a sum of squared Fourier coefficients so that we can apply the Beckner-Bonami inequality. 
The identity \eqref{eqn|integrationByParts} applied with $\varphi=\partial_i f^{\tau}$ and $T\subseteq W\setminus\{i\}$ gives
\begin{align*}
\mathbb E\left[\partial_i f^{\tau}\cdot r_T\right]&=\sqrt{p(1-p)}^{|T|}\cdot \mathbb E\left[\partial_T\partial_i f^{\tau}\right],
\end{align*}
hence
\begin{align}
\left(\mathbb E\left[\partial_T\partial_i f^{\tau}\right]\right)^2 &=\frac{\left(\mathbb E\left[\partial_i f^{\tau}\cdot r_T\right]\right)^2}{(p(1-p))^{|T|}}.
\label{eqn|conversionToFourier}
\end{align}
We now use \eqref{eqn|BecknerBonami}. Let us re-state it in terms of a general random variable $\varphi$: If $a_T=\mathbb E[\varphi r_T]$, then for fixed $j$
\begin{align}
\sum_{|T|=j} a_T^2 &\leq (q-1)^j\theta^{2j}\cdot \|\varphi\|_{q'}^2, \label{eqn|BecknerBonamiVarphi}
\end{align}
where $\theta=\frac1{\sqrt{p(1-p)}}$ and $q'$ is the conjugate of $q$. Taking $\varphi=\partial_i f^{\tau}$ and applying 
\eqref{eqn|conversionToFourier} with $j$ replaced by $m-1$ implies 
\begin{align}
\sum_{T:i\not\in T, |T|=m-1} \left(\mathbb E\left[\partial_T\partial_i f^{\tau}\right]\right)^2
&=\frac{1}{(p(1-p))^{m-1}}\sum_{|T|=m-1} \left(\mathbb E\left[\partial_i f^{\tau}\cdot r_T\right]\right)^2.
\nonumber 
\end{align}
We now use the inequality \eqref{eqn|BecknerBonamiVarphi} to obtain
\begin{align}
\sum_{T:i\not\in T, |T|=m-1} \left(\mathbb E\left[\partial_T\partial_i f^{\tau}\right]\right)^2
&\leq \frac{(q-1)^{m-1}\theta^{2(m-1)}}{(p(1-p))^{m-1}}\cdot \|\partial_i f^{\tau}\|_{q'}^2.
\nonumber
\end{align}
Substituting $\theta=(p(1-p))^{-1/2}$ gives us
\begin{align}
\sum_{T:i\not\in T, |T|=m-1} \left(\mathbb E\left[\partial_T\partial_i f^{\tau}\right]\right)^2
&\leq\frac{(q-1)^{m-1}}{(p(1-p))^{2(m-1)}}\cdot \|\partial_i f^{\tau}\|_{q'}^2.
\label{eqn|innerSumBound}
\end{align}
Summing \eqref{eqn|innerSumBound} over $i$ and substituting into \eqref{eqn|beforeBecknerBonami},
\begin{align}
\Sigma_m(f^{\tau})&\leq \frac{(q-1)^{m-1}}{(p(1-p))^{2(m-1)}}\sum_i \|\partial_i f^{\tau}\|_{q'}^2 \nonumber \\
&=\frac{(q-1)^{m-1}}{(p(1-p))^{2(m-1)}}\sum_i\left(\mathbb E\left[|\partial_i f^{\tau}|^{q'}\right]\right)^{2/q'} \nonumber \\
&\leq \frac{(q-1)^{m-1}}{(p(1-p))^{2(m-1)}}\cdot (b-a)^2\sum_i\left(\mathbb P\left( \partial_i f^{\tau}\neq 0\right)\right)^{2/q'}.
\label{eqn|afterBecknerBonami}
\end{align}
We now apply \eqref{eqn|BoundMVTorus} to the right-hand side of \eqref{eqn|afterBecknerBonami} to obtain
\begin{align}
\Sigma_m(f^{\tau})&\leq
\frac{(q-1)^{m-1}}{(p(1-p))^{2(m-1)}}\cdot (b-a)^2\cdot \left(\frac{b}{apn^{d-1}}\right)^{2/q'}\sum_i 1\nonumber\\
&=
\frac{(q-1)^{m-1}}{(p(1-p))^{2(m-1)}}\cdot (b-a)^2\cdot \left(\frac{b}{apn^{d-1}}\right)^{2/q'}n^d.\label{eqn|almostFinal}
\end{align}
If we choose $q=\frac{2}{1-\zeta}$ and $q'=\frac{2}{1+\zeta}$, then $q-1=\frac{1+\zeta}{1-\zeta}$ and $2/q'=1+\zeta$, so \eqref{eqn|almostFinal} becomes 
\begin{align}
\Sigma_m(f^{\tau})&\leq 
\hat C(\zeta,m,p,a,b)\cdot n^{1+\zeta-\zeta d},\quad\text{ where }\label{eqn|BoundFourierLevelKTorus}\\
\hat C(\zeta,m,p,a,b)&=\left(\frac{1+\zeta}{(1-\zeta)(p(1-p))^2}\right)^{m-1}\cdot (b-a)^2\cdot \left(\frac{b}{ap }\right)^{1+\zeta}.\nonumber
\end{align}
Summing $(p(1-p))^m$ times \eqref{eqn|BoundFourierLevelKTorus} over $m\in\{1,2,\dots,k-1\}$ gives \eqref{eqn|SumLowerFourierLevels}.
\end{proof}

\subsection{Corollary \ref{thm|Consequence2} and related conjectures}\label{subs|CorollaryAndConjectures}
\begin{proof}[Proof of Corollary \ref{thm|Consequence2}] 
The denominator in \eqref{eqn|BoundRemainder} contains $\frac{\|\partial_Mf\|_2}{\|\partial_Mf\|_1}$. We find a lower bound for this component using Cauchy's inequality and \eqref{eqn|PartialNonZero}. 
\begin{align*}
\|\partial_Mf\|_1&=\|\partial_M f\cdot 1_{\partial_Mf\neq 0}\|_1\leq 
\|\partial_M f\|_2\cdot \sqrt{\mathbb P\left(\partial_Mf\neq 0\right)}\\
&\leq \|\partial_M f\|_2\cdot \frac{\sqrt {\theta}}{n^{\frac{d-1}2}}.
\end{align*}
Hence, $\log \frac{\|\partial_Mf\|_2}{\|\partial_Mf\|_1}\geq C_1 \log n$, for some constant $C_1$.

We can sharpen the bound \eqref{eqn|SumLowerFourierLevels} for $k=2$ in the following way  
\begin{align*}\sum_{S\subseteq W;  |S|=1}(p(1-p))^{|S|}\left(\mathbb E\left[\partial_Sf\right]\right)^2 &=
(p(1-p))\sum_{i\in W}\left(\mathbb E\left[\partial_if\right]\right)^2
\\&\leq (p(1-p))\sum_{i\in W}\left((b-a)\mathbb P(A_i)\right)^2.\end{align*}
The total number of edges in the graph is at most $d\nu n^d$ for some constant $\nu$, and \eqref{eqn|BoundMVTorus} gives $\mathbb P(A_i)\leq \frac{b}{apn^{d-1}}$. Hence
\begin{align*}
(p(1-p))\sum_{i\in W}\left((b-a)\mathbb P(A_i)\right)^2 &\leq p(1-p)(b-a)^2\cdot \left(\frac{b}{ap}\right)^2 \cdot\nu\cdot d\cdot n^{d}
\\
&\quad
\times  \left(\frac1{n^{d-1}}\right)^2\\
&=p(1-p)(b-a)^2\cdot\left(\frac{b}{ap}\right)^2\cdot\nu\cdot d\cdot \frac1{n^{d-2}}.
\end{align*}

In dimension $d\geq 2$ this last quantity is bounded by a constant $C'$ independent of $n$, which gives the additive constant in the first of the two bounds in \eqref{eqn|convN}. For the second bound, 
observe that $\|\partial_Mf\|_2^2$ satisfies
\begin{align*}\|\partial_Mf\|_2^2&=\mathbb E\left[|\partial_Mf|^2\right]=\mathbb E\left[|\partial_Mf|^2\cdot 1_{\{\partial_Mf\neq 0\}}\right]\\
&\leq2^{2|M|}(b-a)^2\mathbb P(\partial_M f\neq 0).
\end{align*} 
Therefore, the summation in \eqref{eqn|BoundRemainder}
can be bounded by $\frac{C}{(\log n)^2}\mathbb E\left[N_2\right]$.

 It remains to notice that the numerator of each fraction in the summation contains 
$\|\partial_Mf\|_2^2$ which can be bounded from above by $C_2 \mathbb P(\partial_Mf\neq 0)$. Therefore, the right-hand side of \eqref{eqn|BoundRemainder} is bounded by 
\begin{align*}\frac{C}{\log^2n}\sum_{|M|=2}\mathbb P\left(\partial_Mf\neq 0\right)&= \frac{C}{\log^2n}\mathbb E\left[\sum_{|M|=2}1_{\partial_Mf\neq 0}\right],
\end{align*}
and the last summation in the expected value is precisely the random variable $N_2$. 
\end{proof}

As mentioned earlier, we believe that the following conjectures are true, but we do not know how to prove them. 
\begin{conjecture}
In first passage percolation model, there exists a constant $C$ independent on $N$ such that 
\begin{align*}
\sum_{M\subseteq W_n, |M|=2} \|\partial_Mf\|_2^2&\leq C\cdot N.
\end{align*}
\end{conjecture}

\begin{conjecture}
In first passage percolation model, there exists a constant $C$ independent on $N$ such that 
\begin{align*}
\sum_{M\subseteq W_n, |M|=2} \|\partial_Mf^{\tau}\|_2^2&\leq C\cdot N.
\end{align*}
\end{conjecture}

\section{Relationship between influential and essential edges} \label{sec|relationshipIE}
The constructions in this section are carried out for the general first-passage percolation model $f$, with fixed endpoints $0$ and $nv$. Analogous constructions can be carried out on the torus.
We established in \eqref{eqn:EAinclusion} that $E_i\subseteq A_i$. The following proposition shows that the reverse inclusion does not hold.
\begin{proposition}\label{thm|NoOppositeInclusions}
There exists $n_0$ such that for all $n\geq n_0$ there exists an edge $i$ for which the following holds
\begin{align}
A_i \setminus E_i &\neq  \emptyset, \label{eqn|ReverseEAInclusion} \\
E_i\setminus \left(A_i\cap\left\{\omega_i=a\right\}\right)& \neq \emptyset.  \label{eqn|ReverseAEInclusion}
\end{align} 
\end{proposition}

\begin{proof}
Let us first construct an environment $\omega$ in $E_i\cap \{\omega_i=b\}$. This will be a sufficient example to prove (\ref{eqn|ReverseAEInclusion}). 
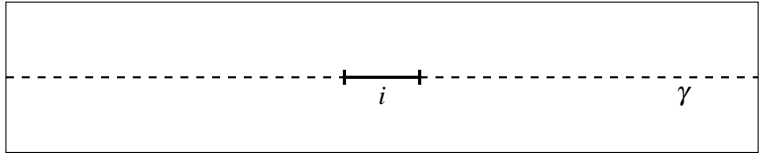
 \begin{figure}[t]
 \centering 
\begin{tikzpicture}
\draw (0,0) -- (10,0) -- (10,2) -- (0,2) -- cycle;
\draw[thick,dashed] (0,1) -- (4.5,1); 
\draw[thick,dashed] (5.5,1) -- (10,1);
\draw[very thick] (4.5,1) -- (5.5,1); 
\node[below] at (9,1) {$\gamma$};
\node[below] at (5,1) {$i$};
\draw[very thick] (4.5,0.9) -- (4.5,1.1);
\draw[very thick] (5.5,0.9) -- (5.5,1.1);
\end{tikzpicture}
\caption{Example of an environment that proves \eqref{eqn|ReverseAEInclusion}.}
 \label{fig|ESetminusAintOmegaia}
\end{figure}

Let us take the straight line $\gamma$ in the graph. Let us pick one edge on this line $\gamma$ and call it $i$. Set $\omega_k$ to be $b$ for $k=i$ and for $k$ outside $\gamma$. Set $\omega_l$ to be $a$ for every edge $l$ on the line $\gamma$ that is different from $i$. Then $\gamma$ is the only geodesic. It passes through $i$ although $\omega_i=b$.  

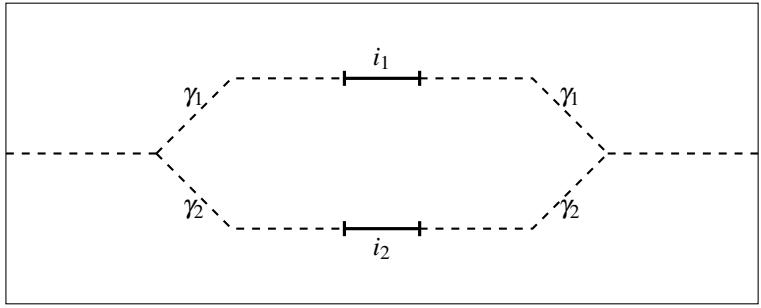
 \begin{figure}[t]
 \centering
\begin{tikzpicture}
\draw (0,-2) -- (10,-2) -- (10,2) -- (0,2) -- cycle;
\draw[thick,dashed] (0,0) -- (2,0);
\draw[thick,dashed] (8,0) -- (10,0);
\draw[thick,dashed] (2,0) -- (2,-1) -- (4.5,-1);
\draw[very thick] (4.5,-1) -- (5.5,-1);
\draw[very thick] (4.5,-1.1) -- (4.5,-0.9);
\draw[very thick] (5.5,-1.1) -- (5.5,-0.9);
\draw[thick,dashed] (5.5,-1) -- (8,-1) -- (8,0);
\draw[thick,dashed] (2,0) -- (2,1) -- (4.5,1);
\draw[very thick] (4.5,1) -- (5.5,1);
\draw[very thick] (4.5,0.9) -- (4.5,1.1);
\draw[very thick] (5.5,0.9) -- (5.5,1.1);
\draw[thick,dashed] (5.5,1) -- (8,1) -- (8,0);
\node[above] at (2.5,1.0) {$\gamma_1$};
\node[below] at (2.5,-1.0) {$\gamma_2$};
\node[above] at (7.5,1.0) {$\gamma_1$};
\node[below] at (7.5,-1.0) {$\gamma_2$};
\node[above] at (5,1) {$i_1$};
\node[below] at (5,-1) {$i_2$};
\end{tikzpicture}
\caption{Example of an environment that proves \eqref{eqn|ReverseEAInclusion}.}
 \label{fig|ASetminusE}
\end{figure}

We now construct an environment $\omega$ in $A_i\setminus E_i$.
Let us pick two paths $\gamma_1$ and $\gamma_2$ that have the same starting points and the same ending points. However, the paths $\gamma_1$ and $\gamma_2$ have sections that are reflections of each other, as shown in the Figure \ref{fig|ASetminusE}. We identify two edges $i_1\in\gamma_1\setminus \gamma_2$ and $i_2\in \gamma_2\setminus \gamma_1$ that are far away from each other. Consider the environment $\omega$ that has the value $b$ on the edges $i_1$ and $i_2$ and on every edge outside of $\gamma_1\cup \gamma_2$. The environment $\omega$ has the value $a$ on each edge from $\gamma_1\cup\gamma_2\setminus\{i_1,i_2\}$. 
The edges $i_1$ and $i_2$ are influential. However, neither of them is essential, because in the unchanged environment $\omega$, each of $\gamma_1$ and $\gamma_2$ is a geodesic.
\end{proof}

\begin{proposition}\label{thm|ABGeneral}
Assume that $a/(b-a)\not\in\mathbb N$.
Then, there exists an integer $n_0$ and an edge $j$ such that for all $n\geq n_0$, the following holds
\begin{align}
A_j\setminus \hat E_j&\neq \emptyset; \label{eqn|AJSetminusHatEJ}
\\
E_j\setminus \hat A_j&\neq \emptyset. \label{eqn|EJSetminusHatAJ}
\end{align}
\end{proposition}
\begin{proof}
Let us first prove \eqref{eqn|AJSetminusHatEJ}. Let us consider two paths $\gamma_1$ and $\gamma_2$ that have the same starting points and the same ending points, but that contain sections that are sufficiently far away from each other. The passage times are set to $b$ for all edges outside of $\gamma_1$ and $\gamma_2$.
There exist integers $k_a$ and $k_b$ for which 
$ak_a+bk_b\in (0,b-a)$. We can make such choices for passage times on disjoint sections of $\gamma_1$ and $\gamma_2$ such that the difference $T(\gamma_1,\omega)-T(\gamma_2,\omega)$ belongs to the open interval $(0,b-a)$. 

Then, let us identify an edge $j$ on the section of $\gamma_1$ far away from $\gamma_2$ that satisfies $\omega_j=b$. The path $\gamma_2$ is the only geodesic on $\omega$ and the path $\gamma_1$ is the only geodesic on $\sigma_j^a(\omega)$. The edge $j$ is not semi-essential on $\omega$, however it is influential. Hence, $\omega\in A_j\setminus \hat E_j$. 

The proof for \eqref{eqn|EJSetminusHatAJ} is similar. We can take the same construction that we used in the proof of \eqref{eqn|AJSetminusHatEJ}. This time, we identify an edge $j'$ on the section $\gamma_2$ that is far away from $\gamma_1$ and that satisfies $\omega_{j'}=a$. The path $\gamma_2$ is the only geodesic on $\omega$ and the path $\gamma_1$ is the only geodesic on $\sigma_{j'}^b(\omega)$. Therefore, the edge $j'$ is essential on $\omega$. However, the edge is not very influential, because $\partial_{j'}f(\omega)$ is strictly smaller than $b-a$.  
\end{proof}

For a set $V$ of edges, we defined $E_V$ as $\cap_{j\in V} E_j$. Therefore, it makes sense to generalize the concept of essential edge and talk about essential sets of edges. Unfortunately, if we define $A_V=\{\partial_{V}f\neq 0\}$, the fundamental inclusion $E_j\subseteq A_j$ does not generalize to sets with more than one element. Let us consider the case $V=\{v_1,v_2\}$.
The following two propositions imply that $E_V\not\subseteq A_V$ and that $A_V\not\subseteq A_{v_1}\cup A_{v_2}$.

\begin{proposition}\label{thm|entanglementDoesNotImplyInfluentiall}
For sufficiently large $n$, there are edges $v_1$ and $v_2$ for which the following holds
\begin{align}
\left\{\partial_{v_1}\partial_{v_2}f\neq 0\right\}\setminus\left(A_{v_1}\cup A_{v_2}\right) &\neq \emptyset.
\label{eqn|entanglementDoesNotImplyInfluentiall}
\end{align}
\end{proposition}
\begin{proof}
 \begin{figure}[t]
 \centering
\begin{tikzpicture} 
\draw (0,-2) -- (10,-2) -- (10,2) -- (0,2) -- cycle;
\draw[thick,dashed] (0,0) -- (1,0);
\draw[thick,dashed] (9,0) -- (10,0);
\draw[thick,dashed] (1,0) -- (1,-1) -- (9,-1) -- (9,0);
\draw[very thick] (3,-1) -- (3.8,-1);
\draw[very thick] (3.8,-1) -- (4.6,-1); 
\draw[very thick] (3,-1.1) -- (3,-0.9); 
\draw[very thick] (3.8,-1.1) -- (3.8,-0.9);
\draw[very thick] (4.6,-1.1) -- (4.6,-0.9);  
\draw[very thick] (3,1.1) -- (3,0.9); 
\draw[very thick] (3.8,1.1) -- (3.8,0.9);
\draw[very thick] (4.6,1.1) -- (4.6,0.9);  

\draw[thick,dashed] (1,0) -- (1,1) -- (3.8,1);
\draw[very thick]  (3.8,1) -- (4.6,1);

\draw[thick,dashed] (4.6,1) -- (9,1)--(9,0); 

\node[above] at (1.5,1.0) {$\gamma_1$};
\node[below] at (1.5,-1.0) {$\gamma_0$};
\node[above] at (8.5,1.0) {$\gamma_1$};
\node[below] at (8.5,-1.0) {$\gamma_0$};
\node[above] at (3.3,1.1) {$w_1$}; 
\node[above] at (4.2,1.1) {$w_2$};
\node[below] at (3.4,-1) {$v_1$};
\node[below] at (4.2,-1) {$v_2$}; 
\end{tikzpicture}
\caption{Example of an environment that proves \eqref{eqn|entanglementDoesNotImplyInfluentiall}.}
\label{fig|entanglementDoesNotImplyInfluentiall}
\end{figure}

Let us consider two paths $\gamma_0$ and $\gamma_1$ from source to sink that have equal number of edges and that have two sections sufficiently far apart, that are reflections of each other, as shown in the Figure \ref{fig|entanglementDoesNotImplyInfluentiall}. 

Let us identify two edges $v_1$ and $v_2$ on $\gamma_0$ and their reflections $w_1$ and $w_2$ on $\gamma_1$. The environment $\omega$ assigns $b$ to every edge outside of $\gamma_0\cup \gamma_1$ and to the edges $v_1$, $v_2$, and $w_2$. The value $a$ is assigned to $\gamma_0\cup\gamma_1\setminus\{v_1,v_2,w_2\}$. 

Neither $v_1$ nor $v_2$ is influential, because turning either $\omega_{v_1}$ or $\omega_{v_2}$ from $b$ to $a$ would result in both paths $\gamma_0$ and $\gamma_1$ being the geodesics.   

However, $\partial_{v_1}\partial_{v_2}f(\omega)=-(b-a)<0$. Hence, 
$\omega\in \{\partial_{v_1}\partial_{v_2}f <0\}\subseteq \{\partial_{v_1}\partial_{v_2}f\neq 0\}$. 
\end{proof}

\begin{proposition}
For sufficiently large $n$, there exists edges $v_1$ and $v_2$ such that 
\begin{align}
\left(E_{v_1}\cap E_{v_2}\right) \setminus \left\{\partial_{v_1}\partial_{v_2}f\neq 0\right\} &\neq \emptyset. \label{eqn|EkElNotSubsetOfAkl}
\end{align}
\end{proposition}

\begin{proof}
Let us consider a straight line $\gamma$ and let us identify two edges $v_1$ and $v_2$ on the line $\gamma$ (see Figure \ref{fig|EkElNotSubsetOfAkl}).
 We will set the environment $\omega$ to satisfy $\omega_k=b$ for every $k\not \in \gamma$; $\omega_{v_1}=\omega_{v_2}=b$; and $\omega_k=a$ for $k\in \gamma\setminus\{v_1,v_2\}$.
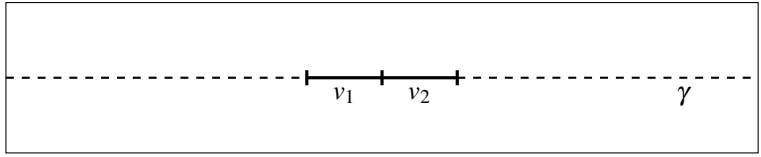
 \begin{figure}[t]
 \centering
\begin{tikzpicture}
\draw (0,0) -- (10,0) -- (10,2) -- (0,2) -- cycle;
\draw[thick,dashed] (0,1) -- (4,1); 
\draw[thick,dashed] (6,1) -- (10,1);
\draw[very thick] (4,1) -- (6,1); 
\node[below] at (9,1) {$\gamma$};
\node[below] at (4.5,1) {$v_1$};
\node[below] at (5.5,1) {$v_2$};
\draw[very thick] (4,0.9) -- (4,1.1);
\draw[very thick] (5,0.9) -- (5,1.1);
\draw[very thick] (6,0.9) -- (6,1.1);
\end{tikzpicture}
\caption{Example of an environment that proves \eqref{eqn|EkElNotSubsetOfAkl}.}
 \label{fig|EkElNotSubsetOfAkl}
\end{figure}

The line $\gamma$ is the geodesic for sufficiently large $n$. Each of the edges $v_1$ and $v_2$ is essential, hence $\omega\in E_{v_1}\cap E_{v_2}$. 
Let $L$ be the number of edges on the line $\gamma$. The values of the function $f$ at the environments 
$\sigma_{v_1}^{\alpha_1}\circ\sigma_{v_2}^{\alpha_2}$ for $(\alpha_1,\alpha_2)\in\{a,b\}^2$ are \begin{align*} f\left(\sigma_{v_1}^b\circ \sigma_{v_2}^b(\omega)\right) &= (L-2)a+2b;\\
f\left(\sigma_{v_1}^a\circ \sigma_{v_2}^b(\omega)\right) &= (L-1)a+b;\\
f\left(\sigma_{v_1}^b\circ \sigma_{v_2}^a(\omega)\right) &= (L-1)a+b;\\
f\left(\sigma_{v_1}^a\circ \sigma_{v_2}^a(\omega)\right) &= La.
\end{align*}
Therefore, the value of $\partial_{\{v_1,v_2\}}f(\omega)$ is $0$ and $\omega\not\in \{\partial_{v_1}\partial_{v_2}f\neq 0\}$.
\end{proof}

\section{Theorems and conjectures about $L^2$ bounds} \label{sec|L2Bounds}
One possible starting point for establishing $L^2$ bounds for the higher order derivatives is  the formula \eqref{eqn|varianceFormula}. This formula can be used together with an obvious $O(n)$ bound for the variance. 
For example, if we fix $k\geq 2$, we can use 
\eqref{eqn|varianceFormula} to obtain 
\begin{align}Cn&\geq \text{var}(f)\geq \sum_{|S|=k}\left(p(1-p)\right)^k \left(\mathbb E\left[\partial_Sf\right]\right)^2,
\label{eqn|BoundDown01}
\end{align}
for some constant $C$. 
There is an order of $O(n^{kd})$ subsets of cardinality $k$. Many of them will have the environment derivatives equal to $0$: for example those that have elements very far away that would guarantee that no geodesic can go through all of them. However, 
even if we request that elements of $S$ are close to each other, the number of such sets $S$ is of order $O(n^{kd})$. 
Let $\mathcal F_k\subseteq\{T:|T|=k\}$. The elements of $\mathcal F_k$ have cardinalities $k$. Assume that $\mathcal F_k$ is large and that it has $O(n^{kd})$ elements. We can further restrict the summation in \eqref{eqn|BoundDown01} to the family $\mathcal F_k$ and obtain 
\begin{align}Cn&\geq  \sum_{S\in \mathcal F_k}\left(p(1-p)\right)^k \left(\mathbb E\left[\partial_Sf\right]\right)^2 \nonumber \\
&\geq C'\cdot n^{kd}\cdot\min_{S\in\mathcal F_k}\left( \mathbb E\left[\partial_Sf\right]\right)^2.
\label{eqn|BoundDown02}
\end{align}
The inequality \eqref{eqn|BoundDown02} now implies the following result.
\begin{proposition}\label{thm|EasyBound} If $\mathcal F_k\subseteq\{T:|T|=k\}$ and if the number of elements in the family $\mathcal F_k$ is larger than or equal to $C\cdot n^{kd}$, for some fixed constant $C$, then there exists a constant $D$ such that 
\begin{align*}
\min_{S\in\mathcal F_k}\mathbb E\left[\partial_Sf\right] &\leq D\cdot\frac1{n^{\frac{kd-1}2}}.
\end{align*}
 \end{proposition}

 \begin{conjecture}\label{conj|Equivalence} Assume that $k\geq 2$ is fixed. If $S$ is a set of edges cardinality $k$ such that $\mathbb E[\partial_Sf]\neq 0$. Then there exists a family $\mathcal F$ of $k$-element subsets of edges that contains $S$ and a constant $C>1$ such that $|\mathcal F|\geq Cn^{dk}$ and 
\[\left(\mathbb E\left[\partial_{S'}f\right]/\mathbb E\left[\partial_Sf\right]\right)^2\in\left(\frac1C,C\right),\] for every $S'\in \mathcal  F$.
\end{conjecture}
We have seen that higher order derivatives can be negative. However, constructing environments with negative derivatives is harder than constructing environments with positive derivatives. So, it is natural to conjecture that $\mathbb E[\partial_Sf]\geq 0$. Unfortunately, this is not true for some sufficiently strange sets $S$ as the following propositions shows.

\begin{proposition}\label{thm|Parasite}
It is possible to choose the parameters $a$ and $b$ and the set $S$ of cardinality $2$ such that $\mathbb E[\partial_Sf]<0$.
\end{proposition}
\begin{proof}
Assume that $a$ and $b$ satisfy $b<3a$. Let $i$ be the edge that connects the starting point $(0,0,\dots, 0)$ with $(1,0,\dots, 0)$. Let $j$ be 
the edge that connects $(1,0,\dots, 0)$ with $(2,0,\dots, 0)$. For this choice of $i$ and $j$, we will prove that 
\begin{align}
&\left(\forall \omega\in\Omega\right) \quad \partial_{i,j}f(\omega)\leq 0 \quad\text{and} 
\label{eqn|parasiteDirection01}\\
&\left(\exists\hat\omega\in\Omega\right) \quad \partial_{i,j}f(\hat\omega)<0. \label{eqn|parasiteDirection02}\end{align} 
We will first prove 
\eqref{eqn|parasiteDirection01}. 
We will first establish the following implication:
\begin{align} 
\sigma_{i,j}^{(b,a)}(\omega)\not\in \hat E_i^C\cap E_j&\Longrightarrow  \partial_{i,j} f\leq 0.\label{eqn|LemmaParasite02}
\end{align}
The implication holds for every choice of $i$ and $j$. 
If $\sigma_{i,j}^{(b,a)}(\omega)\in E_j^C$, then Proposition \ref{thm|forComputer} implies that $f(\sigma_{i,j}^{(b,b)}(\omega))=f (\sigma_{i,j}^{(b,a)}(\omega) )$, hence 
\begin{align*}
\partial_{i,j}f(\omega)&=-\left(f(\sigma_{i,j}^{a,b}(\omega))-f(\sigma_{i,j}^{a,a}(\omega))\right)\leq 0.
\end{align*}
Assume, now, that  $\sigma_{i,j}^{(b,a)}(\omega)\in \hat E_i$. The proposition \ref{thm|forComputer} implies \[f(\sigma_{i,j}^{(b,a)}(\omega))=f (\sigma_{i,j}^{(a,a)}(\omega) )+(b-a),\] hence 
\begin{align*}
\partial_{i,j}f(\omega)&= \left(f(\sigma_{i,j}^{b,b}(\omega))-f(\sigma_{i,j}^{a,b}(\omega))\right)-(b-a)\leq 0.
\end{align*}
This completes the proof of the implication \eqref{eqn|LemmaParasite02}.  
Let $\omega\in\Omega$. We may assume that  
$\sigma_{i,j}^{(b,a)}(\omega)\in \hat E_i^C\cap E_j$, as otherwise \eqref{eqn|LemmaParasite02} would immediately imply \eqref{eqn|parasiteDirection01}.  
Let $\gamma$ be any geodesic on $\sigma_{i,j}^{(b,a)}(\omega)$. 
It must go through $j$ and omit $i$. This geodesic starts at the origin and must 
reach the point $(1,0,\dots, 0)$ so it can go through $j$. It must go through at least $3$ edges. If they all have values $a$, the passage time is at least $3a$. However, if we change this section of the path and replace it with the edge $i$, then the passage time will become $b$, instead of at least $3a$. We assumed that $b<3a$. This contradicts the fact that $\gamma$ is a geodesic. 

We now need to prove \eqref{eqn|parasiteDirection02}--that there is an environment $\hat \omega$ such that $\partial_{i,j}f(\hat \omega)<0$. 
 \begin{figure}[t]
 \centering
\begin{tikzpicture}
\draw[thick,dotted] (1,3) -- (3,3);
\draw[thick] (3,3) -- (9,3);
\draw[dashed] (1,3) -- (1,5) -- (2,5);
\draw[thick] (2,5)--(3,5);
\draw[dashed] (3,5) -- (9,5) -- (9,1) -- (3,1) -- (3,3); 
\node[above] at (2.5,5) {$\xi$};
\node[below] at (1.5,3) {$i$};
\node[below] at (2.5,3) {$j$};
\node at (1,3) {$\bullet$}; 
\node at (2,3) {$\bullet$}; 
\node at (3,3) {$\bullet$};
\node at (9,3) {$\bullet$}; 
\node at (1,5) {$\bullet$}; 
\node at (2,5) {$\bullet$}; 
\node at (3,5) {$\bullet$}; 
\node at (9,5) {$\bullet$}; 
\node at (9,1) {$\bullet$}; 
\node at (3,1) {$\bullet$};
\node[left] at (1,3) {$0$};
\node[right] at (9,3) {$V$};
\node[left] at (1,5) {$L_1$};
\node[right] at (9,5) {$R_1$};
\node[left] at (3,1) {${L_2}$};
\node[right] at (9,1) {${R_2}$};
\node[below] at (5.5,5.0) {$\gamma_1$};
\node[above] at (5.5,1.0) {$\gamma_2$};
\end{tikzpicture}
\caption{Example of an environment that proves \eqref{eqn|parasiteDirection02}.}
 \label{fig|parasiteDirection02}
\end{figure}

Let $V$ denote the terminal point with coordinates $(n,0,\dots, 0)$.  
Let $\gamma_1$ be the path from the origin to $V$ that consists of the following three line segments:
\begin{itemize}
\item The first segment connects the origin with the point $L_1$ whose coordinates are $(0$, $10$, $0$, $\dots$, $0)$;
\item The second segment connects the point $L_1$ with the point $R_1$ with coordinates $(n$, $10$, $0$, $\dots$, $0)$;
\item The third segment connects $R_1$ with the terminal point $V$. 
\end{itemize} 
Let $\gamma_2$ be the path from the origin to $V$ that consists of the following four segments:
\begin{itemize} 
\item The first segment consists of the edges $i$ and $j$;
\item The second segment connects the right endpoint of $j$ with the point $L_2$ 
whose coordinates are $(2,-10,0,\dots, 0)$;
\item The third segment connects the point $L_2$ with the point $R_2$ whose coordinates are $(n,-10,0,\dots, 0)$;
\item The fourth segment connects the point $R_2$ with the point $V$.
\end{itemize}
The environment $\hat\omega$ assigns $b$ to every edge outside of $\gamma_1\cup \gamma_2$. 
Let us identify one edge $\xi$ on the section of the path $\gamma_1$ between $L_1$ and $R_1$ and assign the value $b$ to the edge $\xi$. Every edge on $\gamma_1\setminus\{\xi\}$ and every edge on $\gamma_2\setminus\{i,j\}$ is assigned the value $a$.

The passage time over $\gamma_1$ is 
\[T(\gamma_1,\hat\omega)=(20+n-1)a+b=na+19a+b.\]

The passage time over the straight line that connects the origin with $V$ is larger than $na+19a+b$ for sufficiently large $n$. Indeed, regardless of the choice of 
$\overrightarrow\alpha\in\{a,b\}^2$, 
the passage time over that straight line is greater than or equal to $2a+(n-2)b$ on each of $\sigma_{i,j}^{\overrightarrow \alpha}(\hat\omega)$. The number $2a+(n-2)b$ is larger than $na+19a+b$ whenever \[n>\frac{17a+3b}{b-a}.\]
Let us now evaluate the passage times over $\gamma_2$. 
\begin{align*}
T\left(\gamma_2,\sigma^{(a,a)}_{i,j}(\hat\omega)\right) &= na+20a;\\
T\left(\gamma_2,\sigma^{(a,b)}_{i,j}(\hat\omega)\right) &= na+19a+b;\\
T\left(\gamma_2,\sigma^{(b,a)}_{i,j}(\hat\omega)\right) &= na+19a+b;\\
T\left(\gamma_2,\sigma^{(b,b)}_{i,j}(\hat\omega)\right) &= na+18a+2b.
\end{align*}
On each of the environments $\sigma_{i,j}^{\overrightarrow \alpha}(\hat\omega)$, the smallest passage time is the minimum of the passage times over $\gamma_1$ and $\gamma_2$. Any path that uses an edge outside $\gamma_1 \cup \gamma_2$ pays $b$ on that edge; the straight line from the origin to $V$ has been ruled out above, and any other path that leaves $\gamma_1 \cup \gamma_2$ has length at least $n$ and uses at least as many $b$-valued edges as the straight line, so it is at least as long.
 On the environment $\sigma^{(a,a)}_{i,j}(\hat\omega)$ the minimum is over the path $\gamma_2$ and the value of the first-passage percolation $f$ is $na+20a$. On all other environments, the value of $f$ is $na+19a+b$. Therefore, 
\begin{align*}
\partial_{i,j}f(\hat\omega) &= na+19a+b-2(na+19a+b)+na+20a=-(b -a).
\end{align*}
Thus, $\partial_{i,j}f(\hat \omega)$ is strictly smaller than $0$. 
Since the probability space is finite, we are allowed to conclude that $\mathbb E[\partial_{i,j}f]<0$. 
\end{proof}

It is also possible to find sets $S$ for which $\partial_Sf$ have positive expected values. 

\begin{proposition}\label{thm|NonParasite}
It is possible to choose the parameters $a$ and $b$ and the set $S$ of cardinality $2$ such that $\mathbb E[\partial_Sf]>0$.
\end{proposition}
\begin{proof}
Assume that $a$ and $b$ satisfy $b<3a$. Let $i$ be the edge that connects the starting point $(0,0,\dots, 0)$ with $(0,1,\dots, 0)$. Let $j$ be 
the edge that connects the starting point with $(0,-1,\dots, 0)$. For this choice of $i$ and $j$, we will prove that 
\begin{align}
&\left(\forall \omega\in\Omega\right) \quad \partial_{i,j}f(\omega)\geq 0 \quad\text{and} \label{eqn|nonParasiteDirection01}\\
&\left(\exists\hat\omega\in\Omega\right) \quad \partial_{i,j}f(\hat\omega)>0. \label{eqn|nonParasiteDirection02}\end{align} 
We will first prove 
\eqref{eqn|nonParasiteDirection01}. 
We will first establish the following implication:
\begin{align} 
\sigma_{i,j}^{(a,a)}(\omega)\not\in E_i\cap  E_j&\Longrightarrow  \partial_{i,j} f\geq 0. \label{eqn|LemmaNonParasite01}
\end{align}
If we assume that $\sigma_{i,j}^{(a,a)}(\omega)\in E_i^C$, then the proposition 
\ref{thm|forComputer} implies  
\[f(\sigma_{i,j}^{(b,a)}(\omega))=f(\sigma_{i,j}^{(a,a)}(\omega)).\] Therefore, 
\begin{align*}
\partial_{i,j}f(\omega)&= f(\sigma_{i,j}^{(b,b)}(\omega))-f(\sigma_{i,j}^{(a,b)}(\omega))\geq 0.
\end{align*}
In an analogous way we prove that $\sigma_{i,j}^{(a,a)}(\omega)\in  E_j^C$ implies $\partial_{i,j}f(\omega)\geq 0$. This completes the proof of \eqref{eqn|LemmaNonParasite01}.

Let us now assume that $\omega$ is an environment such that 
$\sigma_{i,j}^{(a,a)}(\omega)\in E_i\cap E_j$. Then, each geodesic from the origin to the terminal vertex must go through both $i$ and $j$. This is impossible, since both edges are incident to the origin, so a path through both would revisit the origin. This completes the proof of 
\eqref{eqn|nonParasiteDirection01}.

Let us prove \eqref{eqn|nonParasiteDirection02}. 
 \begin{figure}[t]
 \centering
\begin{tikzpicture}
\draw[thick,dotted] (1,4) -- (1,2)  ;
\draw[dashed] (1,4) -- (1,5) -- (9,5) -- (9,1) -- (1,1) -- (1,2);
\draw[thick] (1,3) -- (9,3);
\node[right] at (1,3.5) {$i$};
\node[right] at (1,2.5) {$j$};
\node at (1,3) {$\bullet$};  
\node at (9,3) {$\bullet$}; 
\node at (1,5) {$\bullet$};  
\node at (1,4) {$\bullet$}; 
\node at (1,2) {$\bullet$}; 
\node at (9,1) {$\bullet$}; 
\node at (1,1) {$\bullet$};
\node[left] at (1,3) {$0$};
\node[right] at (9,3) {$V$};
\node[left] at (1,5) {$L_1$};
\node[right] at (9,5) {$R_1$};
\node[left] at (1,1) {${L_2}$};
\node[right] at (9,1) {${R_2}$};
\node[below] at (5,5.0) {$\gamma_1$};
\node[above] at (5,1.0) {$\gamma_2$};
\end{tikzpicture}
\caption{Example of an environment that proves \eqref{eqn|nonParasiteDirection02}.}
 \label{fig|nonParasiteDirection02}
\end{figure}
Let $V$ denote the terminal point with coordinates $(n,0,\dots, 0)$.  
Let $\gamma_1$ be the path from the origin to $V$ that consists of the following three line segments:
\begin{itemize}
\item The first segment connects the origin with the point $L_1$ whose coordinates are $(0$, $2$, $0$, $\dots$, $0)$;
\item The second segment connects the point $L_1$ with the point $R_1$ with coordinates $(n$, $2$, $0$, $\dots$, $0)$;
\item The third segment connects $R_1$ with the terminal point $V$. 
\end{itemize} 
Let $\gamma_2$ be the path from the origin to $V$ that consists of the following three segments:
\begin{itemize} 
\item The first segment connects the origin with the point $L_2$ whose coordinates are $(0$, $-2$, $0$, $\dots$, $0)$;
\item The second segment connects the point $L_2$ with the point $R_2$ with coordinates $(n$, $-2$, $0$, $\dots$, $0)$;
\item The third segment connects $R_2$ with the terminal point $V$. 
\end{itemize}
The environment $\hat\omega$ assigns $b$ to every edge outside of $\gamma_1\cup \gamma_2$ and the value $a$ to every edge on 
$\gamma_1\cup \gamma_2\setminus\{i,j\}$. 
We can choose $n$ to be large enough that the straight line segment from $0$ to $V$ is not a geodesic on any of the environments 
$\sigma_{i,j}^{\overrightarrow \alpha}(\hat\omega)$ for $\overrightarrow \alpha\in\{a,b\}^2$. 
 The passage time over the segment between $0$ and $V$ is 
$nb$. The passage time over each of $\gamma_1$ and $\gamma_2$ is either $(n+4)a$ or $(n+3)a+b$. These will be smaller than 
$nb$ if we choose \[n>\frac{3a+b}{b-a}.\]
The first-passage percolation times on $\sigma_{i,j}^{(a,a)}(\hat\omega)$, $\sigma_{i,j}^{(a,b)}(\hat\omega)$, and 
$\sigma_{i,j}^{(b,a)}(\hat\omega)$ are $(n+4)a$. The first-passage percolation time on $\sigma_{i,j}^{(b,b)}(\hat\omega)$
is $(n+3)a+b$. Therefore, 
\begin{align*}
\partial_{i,j}f(\hat\omega)&=(n+3)a+b-2\cdot (n+4)a+(n+4)a=b-a>0.
\end{align*}
This completes the proof of \eqref{eqn|nonParasiteDirection02} and, hence, of the proposition.
\end{proof}

We believe that there are some reasonable conditions on $S$ that avoid sets like those in proposition \ref{thm|Parasite}. 
\begin{conjecture}\label{conj|PositiveExpectation} For well-behaved sets $S$, the expected value of $\partial_Sf$ is non-negative. 
\end{conjecture}

Moreover, it is likely that the following conjecture is true.
 \begin{conjecture}\label{conj|PositiveDrift} Assume that $k\geq 2$ is fixed. There is a real number $\theta_k\in(0,1)$ such that for most reasonable sets $S$ of edges  of cardinality $k$ and every positive $m$, \[\mathbb P(\partial_Sf=-m)\leq  \theta_k \mathbb P(\partial_Sf=m).\] 
\end{conjecture}
The conjecture \ref{conj|PositiveDrift} would imply that $\mathbb E\left[\partial_Sf\right]$ is comparable to $\mathbb P(\partial_Sf\neq 0)$, i.e. that there is a constant $C_k>1$ such that $\mathbb E\left[\partial_Sf\right]/\mathbb P\left(\partial_Sf\neq 0\right)\in \left(1/C_k,C_k\right)$.

The conjectures \ref{conj|Equivalence} and  \ref{conj|PositiveDrift} and the proposition \ref{thm|EasyBound}  would imply that each of $\mathbb P(\partial_Sf\neq 0)$, $\|\partial_Sf\|_1$, and $\|\partial_Sf\|_2$ is bounded above by $ D n^{-\frac{kd-1}2}$. 
Even these bounds would not be sufficient to improve the variance bounds. For improving the variance in the torus model, the following, harder, conjecture is needed for at least one integer $k\geq 3$. 
\begin{conjecture}
\label{eqn|HardConjecture} 
Fix $k\geq 3$. There exists a constant $\theta\in(0,\frac d2)$ such that for sufficiently general choice of the set $S$ of edges of cardinality $k$ and it subset $S'$ of cardinality $k-1$, the following holds \[n^{\theta} \mathbb P(\partial_Sf\neq 0) \approx  \mathbb P(\partial_{S'}f\neq 0).\]
\end{conjecture}
The conjecture \ref{eqn|HardConjecture} would be a major step towards obtaining the algebraic bound $O(n^{\alpha})$ in dimension $d$ for $\alpha=1+2\theta-d$. Indeed, together with conjectures  \ref{conj|Equivalence} and  \ref{eqn|HardConjecture}, the sum of squares of $L^2$ norms of derivatives of order $k-1$ would be of order at most $O(n^{\alpha})$ because
\begin{align*}
\sum_{|M|=k-1}\|\partial_Mf\|_2^2&\leq n^{(k-1)d}\cdot \left(\frac{D}{n^{\frac{kd-1}2}}\cdot n^{\theta}\right)^2= D^2 n^{kd-d-kd+1+2\theta}=D^2n^{\alpha}.
\end{align*}

 \begin{appendix}

\section{Idempotent functions} 
 \begin{proof}[Proof of Proposition \ref{prop|IdempotentFixedPoints}]
 If $x\in \psi(R)$, then there is $z\in R$ such that $x=\psi(z)$. We now have $\psi(x)=\psi(\psi(z))=\psi(z)=x$, hence $x=\psi(x)$. This proves $\psi(R)\subseteq \left\{x\in R: \psi(x)=x\right\}$. The reverse inclusion is obvious. 
 \end{proof}
 
 \begin{proof}[Proof of Proposition \ref{prop|IdempotentPsiXi}]
 The equality \eqref{eqn|generalResultIdempotentPsiXi} will follow from the following four inclusions 
\begin{align}
&\quad\psi\left(\psi^{-1}(Q)\cap \xi(R)\right)\subseteq
\psi(\psi^{-1}(Q))\subseteq Q\cap \psi(R)\subseteq \psi^{-1}(Q)\cap \psi(R)\nonumber \\
&\subseteq \psi\left(\psi^{-1}(Q)\cap \xi(R)\right). 
\label{eqn|ThreeInclusions}
\end{align}
The first two of the inclusions hold for all functions $\psi$ and $\xi$, not just the idempotent ones.
Indeed, $\psi^{-1}(Q)\cap\xi(R)\subseteq \psi^{-1}(Q)$ implies the first inclusion; the relation $\psi(\psi^{-1}(Q))\subseteq Q$ holds for all $\psi$ and $Q$, while $\psi^{-1}(Q)\subseteq R$ implies $\psi(\psi^{-1}(Q))\subseteq \psi(R)$.  

Let us now prove the third inclusion. Assume that $x\in Q\cap \psi(R)$. We need to prove that $x\in \psi^{-1}(Q)$. From \eqref{eqn|rangeFixedPoints} we have $x=\psi(x)$ which is sufficient to conclude that 
$\psi(x)\in Q$. 

We now prove the fourth inclusion in \eqref{eqn|ThreeInclusions}, i.e.
\begin{align*}
 \psi^{-1}(Q)\cap \psi(R)&\subseteq  \psi\left(\psi^{-1}(Q)\cap \xi(R)\right).  
\end{align*}
Let $x\in \psi^{-1}(Q)\cap \psi(R)$. We need to prove that there exists an element $y$ of the set $\psi^{-1}(Q)\cap \xi(R)$ such that $\psi(y)=x$. We will prove that we may take $y=\xi(x)$. We must prove that $\psi(y)=x$, 
$y\in \psi^{-1}(Q)$, and $y\in \xi(R)$. 
The last assertion, that $y\in \xi(R)$, is obvious because of our choice $y=\xi(x)$. Since $x\in \psi(R)$, the equality \eqref{eqn|rangeFixedPoints} implies $x=\psi(x)=\psi(\xi(x))=\psi(y)$, which implies the first assertion that $\psi(y)=x$. The second assertion, $y\in \psi^{-1}(Q)$, follows from $\psi(y)=\psi(\xi(x))= \psi(x)\in Q$, which holds due to our assumption $x\in \psi^{-1}(Q)$. 
\end{proof}

\begin{proof}[Proof of Proposition \ref{prop|IdempotentPsiAPsiB}] 
Since $\psi^a(R)\cup \psi^b(R)=R$, it suffices to prove the following four inclusions \begin{align} E\cap \psi^a(R) &\subseteq  (\psi^a)^{-1}(E) , \label{eqn|InclusionAEGenPsi}\\
E\cap \psi^b(R) &\subseteq  (\psi^a)^{-1}(E)  , \label{eqn|InclusionAEGenXi}\\
(\psi^b)^{-1}(\hat E)\cap \psi^a(R)&\subseteq  \hat E,  \label{eqn|InclusionHatAEGenPsi}\\
(\psi^b)^{-1}(\hat E)\cap \psi^b(R)&\subseteq  \hat E. \label{eqn|InclusionHatAEGenXi}
  \end{align}
We will first apply \eqref{eqn|generalResultIdempotentPsiXi} to $\psi=\psi^a$, $\xi=\mathrm{id}$, and $Q=E$. We will not use the first equality, so the choice of $\xi$ is immaterial.
 \begin{align*}E\cap \psi^a(R) & =(\psi^a)^{-1}(E)\cap \psi^a(R)\subseteq (\psi^a)^{-1}(E),\end{align*}
which implies \eqref{eqn|InclusionAEGenPsi}.
Now we use our assumptions $E\subseteq \hat E$ and $(\psi^b)^{-1}(\hat E)\subseteq (\psi^a)^{-1}(E)$. We apply \eqref{eqn|generalResultIdempotentPsiXi} to $\psi=\psi^b$ and $Q=\hat E$. 
\begin{align*}
E\cap \psi^b(R)\subseteq \hat E\cap \psi^b(R)&= (\psi^b)^{-1}(\hat E)\cap \psi^b(R)\subseteq (\psi^b)^{-1}(\hat E)\subseteq (\psi^a)^{-1}(E).
\end{align*}
The last inclusion implies \eqref{eqn|InclusionAEGenXi}.

We now prove \eqref{eqn|InclusionHatAEGenXi}. We apply \eqref{eqn|generalResultIdempotentPsiXi} to $\psi=\psi^b$ and $Q=\hat E$.
\begin{align*}
(\psi^b)^{-1}(\hat E)\cap \psi^b(R)&=\hat E\cap \psi^b(R)\subseteq \hat E.
\end{align*}
For the proof of \eqref{eqn|InclusionHatAEGenPsi}, we use the assumptions $E\subseteq \hat E$ and $(\psi^b)^{-1}(\hat E)\subseteq (\psi^a)^{-1}(E)$ and apply 
\eqref{eqn|generalResultIdempotentPsiXi} to $\psi=\psi^a$ and $Q= E$.
\begin{align*}
(\psi^b)^{-1}(\hat E)\cap \psi^a(R)\subseteq (\psi^a)^{-1}(E)\cap \psi^a(R)&=E\cap \psi^a(R)\subseteq E\subseteq \hat E.
\end{align*}
This completes the proof of \eqref{eqn|InclusionHatAEGenPsi}. 
 \end{proof}

\end{appendix}

\end{document}